\documentclass[11pt]{amsart}
\usepackage{amsmath,amsfonts,amssymb,amsthm,amscd,latexsym,euscript,verbatim}
\usepackage[all]{xy}

\makeatletter
\def\section{\@startsection{section}{1}%
  \z@{1.1\linespacing\@plus\linespacing}{.8\linespacing}%
  {\normalfont\Large\scshape\centering}}
\makeatother

\theoremstyle{plain}

\newtheorem*{conj*}{Root Groups Conjecture}
\newcommand{\etype}[1]{\renewcommand{\labelenumi}{(#1{enumi})}}
\newtheorem*{thm1.2}{(1.2) Theorem}
\newtheorem*{thm1.3}{(1.3) Theorem}
\newtheorem*{thm1.4}{(1.4) Theorem}
\newtheorem*{prop*}{Proposition}

\def\Miy{\operatorname{Cl}}

\newtheorem{prop}{Proposition}[section]
\newtheorem{thm}[prop]{Theorem}

\newtheorem{cor}[prop]{Corollary}
\newtheorem{lemma}[prop]{Lemma}

\theoremstyle{definition}
\newtheorem{Def}[prop]{Definition}
\newtheorem*{Def*}{Definition}

\newtheorem{Defsnot}[prop]{Definitions and notation}
\newtheorem{example}[prop]{Example}
\newtheorem{examples}[prop]{Examples}

\newtheorem*{notation*}{Notation}
\newtheorem{remark}[prop]{Remark}

\newtheorem*{remark*}{Remark}

\newcommand{\calb}{\mathcal{B}}

\newcommand{\calg}{\mathcal{G}}

\newcommand{\calx}{\mathcal{X}}

\newcommand{\ff}{\mathbb{F}}

\newcommand{\uu}{\mathbb{U}}

\newcommand{\Cl}{\operatorname{Cl}}

\newcommand{\ga}{\alpha}
\newcommand{\gb}{\beta}
\newcommand{\gc}{\gamma}

\newcommand{\gd}{\delta}
\newcommand{\one}{\mathbf 1}

\newcommand{\gl}{\lambda}

\newcommand{\gvp}{\varphi}
\newcommand{\gr}{\rho}
\newcommand{\gs}{\sigma}

\newcommand{\gt}{\tau}

\newcommand{\sminus}{\smallsetminus}
\newcommand{\lan}{\langle}
\newcommand{\ran}{\rangle}

\newcommand{\Aut}{{\rm Aut}}

\def\eroman{\etype{\roman}}

\newcommand{\half}{\frac{1}{2}}

\newcommand{\widebar}[1]{\overset{\mskip1mu\hrulefill\mskip1mu}{#1}
                \vphantom{#1}}

\numberwithin{equation}{section}

\begin{document}
{\title[Structure of primitive  axial algebras]{Structure of
primitive axial algebras}
\author[Louis Halle Rowen , Yoav Segev]
{Louis Halle Rowen$^*$\qquad Yoav Segev}

\address{Louis Rowen\\
         Department of Mathematics\\
         Bar-Ilan University\\
         Ramat Gan\\
         Israel}
\email{rowen@math.biu.ac.il}
\address{Yoav Segev \\
         Department of Mathematics \\
         Ben-Gurion University \\
         Beer-Sheva 84105 \\
         Israel}
\email{yoavs@math.bgu.ac.il}}
\thanks{$^*$The first author was supported by the ISF grant 1994/20}

\keywords{axis, flexible algebra, power-associative, fusion rules,
idempotent}
\subjclass[2010]{Primary: 17A05, 17A15, 17A20 ;
 Secondary: 17A36, 17C27}

\begin{abstract}
``Fusion rules'' are laws of multiplication among eigenspaces of an
idempotent. This terminology is relatively new and is closely
related to primitive axial algebras, introduced recently by Hall,
Rehren and Shpectorov. Axial algebras, in turn, are closely related
to $3$-transposition groups and vertex operator algebras.

In  earlier work we studied primitive axial algebras, not
necessarily commutative, and showed that they all have Jordan type.
In this paper, we  show that all finitely generated primitive axial
algebras are direct sums of specifically described flexible finite
dimensional noncommutative algebras, and commutative axial algebras
generated by primitive axes of the same type.  In particular, all
primitive axial algebras are flexible. They also have Frobenius
forms. We give a precise description of all the primitive axes of
axial algebras  generated by two primitive axes.
\end{abstract}
\date{\today}

 \maketitle

{\small\tableofcontents}

\section{Introduction}

 Motivated by group theory and associated schemes, considerable interest has
arisen in studying {\it  axes} (a kind of idempotents) and the
algebras they generate, called {\it  axial algebras}. In previous
papers \cite{RoSe1,RoSe2,RoSe3}, we classified non-commutative
algebras generated by two primitive axes.  Our goal in this paper is
to produce techniques to study the structure of (not necessarily
commutative) primitive axial algebras of Jordan type, utilizing a close study of
algebras generated by two primitive axes.

\begin{Defsnot}\label{not1}$ $
\begin{enumerate}
\item
$A$ always denotes an  algebra {\it (possibly not commutative)}, over a
field~$\ff$ of characteristic $\ne 2.$ We designate an idempotent $a
= a^2\in A$
 and elements
$\gl,\gd\notin \{0,1\}$ in~$\ff.$

\item For $ y\in A,$ write $L_y$ for the left  multiplication map
$z\mapsto yz$ and $R_y$ for the right multiplication map $z\mapsto
zy$.

\item
Write $A_{\eta}(X_a)$  for the $\eta$-eigenspace  of $A$ with
respect to $X_a,$  where $X\in \{R, L\}.$

A {\it left axis} $a$ is a semisimple idempotent, i.e.,
$A$ is a direct sum of its left eigenspaces with respect to $L_a$.
A  left axis $a$ is {\it  primitive} if $A_1(L_a) = \ff a.$
A {\it right axis} is defined similarly.  A {\it (two sided) \it axis} $a$
is a left axis which is also a right axis and such that $L_aR_a=R_aL_a,$
i.e., $a(ya)=(ay)a$ for all $y \in A.$
The axis $a$ is {\it primitive} if it is primitive both as a left and a right axis, i.e.,
\[
A_1(a): = A_{1}(L_a) =\ff a = A_{1}(R_a).
\]

\item
We say that a left axis $a$ has  {\it  eigenvalue set $\tilde \gl$}, if $\tilde \gl = \{ \gl_1, \dots, \gl_m$\}
is the set of eigenvalues
 of $L_a$.  The eigenvalues $0$ and $1$ play a special role, so
 in this paper we take the simplest nontrivial case, that there  is only one left eigenvalue $\gl$ other than $\{0,1\}$;
we say that  $a$ is a left axis of {\it  type $\gl$} if
\begin{equation}\label{dec1}
A=\overbrace{A_{1}(L_a)\oplus A_{0}(L_a)}^{\text {$0$-part}}\oplus
\overbrace{A_{\gl}(L_a)}^{\text{$1$-part}},
\end{equation}
where  this is a  $\mathbb Z_2$-grading of $A,$ which is called the
``{\it fusion rules}'',    a special case of general fusion rules treated
in \cite{DPSC}.
A  {\it    right axis of
 type $\gd$} is defined analogously.

An axis $a$ of type $(\gl,\gd)$ is a left axis of
type $\gl$ which is also a    right axis of type $\gd.$    Since $L_aR_a=R_aL_a,$
$A$ decomposes as a direct sum of subspaces
\[
A_{\mu,\eta}(a):= A_{\mu}(L_a)\cap A_{\eta}(R_a).
\]
 An element in
$A_{\gl,\gd }(a)$ will be called a $(\gl,\gd)$-{\it eigenvector} of
$a,$  with {\it eigenvalue pair} $(\gl,\delta)$.

 If $\gl =\gd$ we
say that $a$ is of {\it commutative} type, which we sometimes call $\gl,$
and sometimes $(\gl,\gl);$
otherwise $a$ is  of {\it noncommutative}  type $(\gl,\gd).$

When $\gl\ne\gd,$ following \cite[Definition~1.3]{RoSe1}, we continue
the decomposition of~\eqref{dec1} to
\begin{equation}\label{dec11}
A=\overbrace{A_{1}(a) \oplus A_{0,0}(a) }^{\text {$(+,+)$-part}}
\oplus \overbrace{A_{0,\gd}(a)}^{\text {$(+,-)$-part}} \oplus
\overbrace{A_{\gl,0}(a)}^{\text {$(-,+)$-part}} \oplus
\overbrace{A_{\gl,\gd}(a).}^{\text {$(-,-)$-part}}
\end{equation}
and this is a noncommutative (i.e.~two sided) $\mathbb Z_2\times
\mathbb Z_2$ -grading of $A.$

\item
Throughout this paper, when we say that $a\in A$ is a primitive axis,
we mean that $a$ has some type $(\gl,\gd);$ that is we assume that
the eigenvalues of $L_a$ are contained in $\{0,1,\gl\},$ that the eigenvalues of $R_a$ are contained in $\{0,1,\gd\},$
and that equation \eqref{dec11} is satisfied.

\item
In \cite[Theorem 2.16]{RoSe3} we proved that if $A$ is generated by
primitive axes, then for all these axes we have
\begin{equation}\label{J}
A_{\gl,0}(a)=A_{0,\gd}(a)=0.
\end{equation}
 Since we will be concerned with algebras generated by primitive axes, from now on, {\it a primitive axis of type
$ (\gl,\gd)$} is a primitive axis satisfying \eqref{J}. Thus writing
$A_0(a)$ for $A_{0,0}(a),$ we have \
\[
A=\overbrace{A_{1}(a)\oplus A_{0}(a)}^{\text {$0$-part}}\oplus \overbrace{A_{\gl}(a)}^{\text{$1$-part}},
\]
with a $ \mathbb Z_2$-grading. Accordingly any
$y\in A$ has a unique decomposition
\[
y=\ga_y a+y_{0}+y_{\gl,\gd}, \qquad \ga_y\in\ff,\, y_0 \in A_0(a), \ y_{\gl,\gd}\in
A_{\gl,\gd}(a).
\]

\item
 $X\subseteq A$ denotes a set of primitive axes, not necessarily
generating $A.$ $\langle\langle X \rangle\rangle$ denotes the
subalgebra of $A$ generated by $X$.

\item
$A$ is a  {\it primitive  axial algebra} (of Jordan type), {\it called a PAJ},
if it is generated by a set $X$ of primitive axes  (not necessarily
of the same type). We say that $A$ is {\it finitely generated} if
$X$ is a finite set, and $A$ is  {\it $n$-generated}  if $|X|=n.$

A~{\it CPAJ} is a commutative PAJ.

\item
Given a PAJ $A$ we denote by $\chi(A)$ the set of {\it all} primitive
axes of $A.$
\end{enumerate}
\end{Defsnot}

\begin{remark}
When a primitive axis $a$ is of  noncommutative type $(\gl,\gd),$
then  $\gd = 1 -\gl,$ by \cite[Theorem 2.5, Proposition~2.9, and
Example~2.6]{RoSe2}.
\end{remark}

\begin{Def}\label{hom}
If $x\in X$ is of type $(\gl , \gd),$ resp.~ $\gl$ in the commutative case, we define the {\it
complementary type} of $x$ to be~$(\gd, \gl),$ resp.~$1-\gl.$

Write $X_{\gl}$ (resp.~$X_{\gl,\gd}$) for the set of primitive
axes of~$X$ of commutative type $\gl$~(resp.~of noncommutative type
$(\gl,\gd)).$ Write $X_{\{\gl,\gd\}}:=X_{\gl,\gd}\cup X_{\gd,\gl}.$ Call  $X$ {\it uniform} of type $\gl$
(resp.~type~$(\gl,\gd),$ resp.~type~$\{\gl,\gd\}$) if $ X=X_\gl$  (resp.~$X=X_{\gl,\gd},$
resp.~$X=X_{\{\gl,\gd\}}$). A PAJ $A$ is {\it uniformly generated}  of type
$\gl$ (resp.~type$(\gl,\gd),$resp.~type~$\{\gl,\gd\}$) if $A = \lan\lan
X_{\gl}\ran\ran$ (resp.~$A=\lan\lan X_{\gl,\gd}\ran\ran,$ resp.~$A = \lan\lan
X_{\{\gl,\gd\}}\cup X_{\{\gd,\gl\}}\ran\ran$).
\end{Def}

By definition, all of the algebras in \cite{HRS} are uniformly
generated of type $\gl.$

\subsection{The structure of the paper}$ $
\medskip

After a review of the basics of PAJ's, we bring in the fundamental
notion of Miyamoto involutions, and the ensuing topology on the set
of primitive axes (\S\ref{Miyfhol0}), extended to the noncommutative setting.
Accompanying this is the axial graph,
whose  components turn out to be uniform, cf.~\S\ref{dePAJ}, of type
which  is uniquely determined up to taking complementary type.

This leads us to the ``axial decomposition'' (Theorem~\ref{miya1})
presenting $A$ as a direct product of a CPAJ and various
noncommutative PAJ's, given in  \S\ref{nonPAJ}, the most basic of
which is  $\uu_E (\gl)$ ($\gl \ne 0,1$),  defined  to be the algebra
having as a basis the set of idempotents $E:= \{ e_i: i \in I\}$,
with multiplication given by $e_ie_j = \gd e_i + \gl e_j,$ for each
$i,j\in I$, where $\gd = 1 - \gl$ (Definition~\ref{uuI}), whose
dimension is $|E|.$ From this we conclude that all PAJ's are
flexible (Theorem~\ref{flex2}).

Having reduced the theory to uniformly generated CPAJ's, in
\S\ref{lowdim} we delve deeper into properties of primitive axes. We
start with the projection $\gvp_a$ of~$A$ to $\ff a$, for a
primitive axis $a$, and continue by determining all idempotents in a
2-generated PAJ $B=\lan\lan a,b\ran\ran$, which all turn out to be
primitive axes in $B$.

In \S\ref{Frob} we obtain a Frobenius form for any PAJ, which is not
quite unique, and implies $A_0(a)^2 \subseteq A_0(a),$ which was an
axiom in \cite{HRS} but is in fact redundant.

\subsection{Main results}$ $
\medskip

$A$ is a  PAJ  generated by a set of primitive axes~$X$, and $a\in
X$ in all of these theorems. The notation and terminology
can be found in the appropriate sections.  These are our main results:
\medskip

{\bf Theorem A} (Theorem~\ref{KMt})
\begin{enumerate}\eroman
\item
(Generalizing \cite[Corollary 1.2, p.~81]{HRS}.) $A$ is spanned by
the set~$\Miy(X)$.

\item
Let $V\subseteq A$ be a subspace containing $X$ such that $xv\in V,$
or $vx\in V,$ for all $v\in V$ and $x\in X,$ then $V=A.$

\item
If $X$ is uniform of type $\{\gl,\gd\}$ (resp.~$(\gl,\gd)$
resp.~$\gl$), then $A$ is spanned by primitive axes of type
$\{\gl,\gd\}$ (resp.~$(\gl,\gd)$ resp.~$\gl$).
\end{enumerate}
\medskip

{\bf Theorem B} (Theorem~\ref{miya1})
Any  PAJ $A=\lan\lan X\ran\ran$ is a sum of uniformly generated algebras $A_i:=\lan\lan X_i\ran\ran,$
where $\{X_i : i\in I\}$ are the connected components of $\Cl(X)$ in the
axial graph (\S 3).

 This is strengthened in
Theorem~\ref{strongcon}.

\medskip

{\bf Theorem C} (Theorem~\ref{U1} and
Theorem~\ref{almostcomm})$ $
\smallskip

\noindent Suppose  $X= X' \cup X''$ with $X' = X_{\gl,\gd}$ and $X''
= X_{\gd,\gl},$ where $\gl + \gd = 1$ and $\gl \ne \gd.$ Set
$A':=\langle\langle X'\rangle\rangle$ and $A'':=\langle\langle
X''\rangle\rangle.$ Let
\[
Z =\Big\{ \sum \ff x'x'': x'\in X',\ x''\in X''\Big\}.
\]
Let $a, a'\in A$ be two primitive axes of type $(\gl,\gd)$
and $b\in A$ be an axis of type $(\gd,\gl).$
\begin{enumerate}\eroman
\item $A'$ is a direct product of copies of various $\uu_{E_i} (\gl)$, and  $A''$ is a direct product of copies of
various $\uu_{E_j
} (\gd)$.
\item
$ab, ba\in
A_{\gl,\gd}(a'),$  if $aa'\ne 0,$ while $ab, ba\in A_0(a'),$ if
$aa'=0.$  In particular if $x'\in A$ is an axis of type $(\gl,\gd),$
with $a'x'\ne 0,$ then $A_0(a')\cap Z=A_0(x')\cap Z$ and
$A_{\gl,\gd}(a')\cap Z=A_{\gl,\gd}(x')\cap Z.$

\item
$Z$ is an ideal in $A,$ with $Z^2 = 0,$ so $A = A' + A'' + Z.$
Consequently $A$ is finite dimensional if $X$ is finite.

\item
  If $w\in A_{\gl,\gd}(a),$ then $w=v'+z',$
with $v'\in(\ff a+A')\cap A_{\gl,\gd}(a)$ and $z'\in Z\cap
A_{\gl,\gd}(a).$

In particular, if $a\in X',$ then
\[
A_{\gl,\gd}(a) =A'_{\gl,\gd}(a)+Z\cap A_{\gl,\gd}(a).
\]
Hence if $a_1, a_2\in X'$ are in the same connected component of
$X',$ then $A_{\gl,\gd}(a_1) =A_{\gl,\gd}(a_2).$  Furthermore $A_{\gl,\gd}(a)^2=0.$

\item
Any primitive axis in $A'$ (resp. $A''$)  is a primitive axis in
$A.$

\item
Suppose $X'$ is the set of all primitive axes of $A$ of type $(\gl,\gd).$
Taking one
 primitive axis $a_j'$ for each connected component of $X'$
 we have $Z\subseteq \sum_{j\in J}A'_{\gl,\gd}(a'_j).$
The same assertion holds with $X''$ in place
of~$X'$ (and $(\gd,\gl)$ in place of $(\gl,\gd)$).

\item
If $X'$ is connected, then $A_{\gl,\gd}(a)=\sum \ff_{x'\in X'} (x'-a)+Z,$ for all $a\in X',$
so $A_{\gl,\gd}(a)$ is an ideal of $A,$ and $A/A_{\gl,\gd}(a) \cong \ff\times \widebar{A''}.$

\item
 Suppose $X=\chi(A).$ Taking one
 primitive axis $a_j'$ for each connected component of $X'$ and  one
 primitive axis $a_j''$ for each connected component of $X''$, let
\[
\textstyle{I:=\sum _j  A'_{\gl,\gd}(a'_j) +\sum _j   A''_{\gl,\gd}(a_j'').}
\]
Then $I$  is an ideal in $A,$ and $A/I$ is
 a direct product of fields.

\item
$A$ has no nonzero annihilating element. Hence $A$ is a direct
product of the $A_i$'s where $A_i$ is as in Theorem B.
 \end{enumerate}
\medskip

{\bf Theorem D} (Theorem \ref{dec2})
Every PAJ $A$ is a direct product of uniformly generated noncommutative PAJ's (described in Theorem C) with a CPAJ.
\medskip

{\bf Theorem E} (Theorem \ref{flex2})
Any PAJ  is flexible.
\medskip

{\bf Theorem F} (Theorem~\ref{thm frob} and Remark \ref{-1}) $ $
\begin{itemize}
\item[(i)]
There is a set  $X'$ of
primitive axes, generating $A$, and a
 Frobenius form for
which all of the  axes of $X'$ have norm $1.$

\item[(ii)]
One of the following holds
\begin{itemize}
\item[(1)]
We can choose $X'$ in (i) so that $a\in X';$ or

\item[(2)]
$a$ is of type $-1,$ there  exists a primitive axis
$b$ of type $2,$ with $ab\ne 0,$ and $a$ is in the radical
of any Frobenius form on $A.$
\end{itemize}
\end{itemize}
\medskip

{\bf Theorem G} (Theorem~\ref{A02})
$A_0(a)^2\subseteq A_0(a),$ unless perhaps $a$ is as in part (ii2)
of Theorem F.
\medskip

{\bf Theorem H} (Theorem \ref{axes}, Theorem~\ref{nothalf}, Proposition~\ref{phalf})$ $
\smallskip

\noindent
If $A= \langle\langle a,b \rangle\rangle$ is a
2-generated CPAJ, then   all nontrivial
idempotents in $A$ are primitive axes, and are explicitly described. For
$\gl \ne \half,$   $A$ contains exactly $6$ idempotents,
and all primitive axes of the same type are conjugate with respect to
a Miyamoto involution. For $\gl = \half,$ let $\ga_b$ as in notation \ref{not1}(6):
\begin{enumerate}
 \item
If $\ga_b\ne 0,$ then a primitive axis  $b$ is conjugate to $a$ with
respect to a Miyamoto involution,  iff $\ff$ contains a square root
of $\ga_b.$

\item
If $\ga_b=0,$ let $E_d=\{a+\gr e\mid \gr\in\ff\},$ for $d\in\{a,\one-a\}.$
Then
\[
E_d=\{e_1^{\gt_e}\mid \text{$e$ a primitive idempotent in A}\},
\]
for any $e_1\in E_d.$
\end{enumerate}

\medskip
In summary, all nontrivial idempotents of a 2-generated CPAJ are
primitive axes. We use this in our next paper to address axes and
conjugacy in arbitrary PAJ's.

\section {Miyamoto involutions}\label{Miyfhol0}

In order to treat the structure of PAJ's in greater depth, we need
another topic.

\subsection{Miyamoto involutions in the noncommutative
setting}\label{Miyfhol}$ $
\smallskip

  It is easy to check
that any $\mathbb Z_2$-grading of $A$ induces an
automorphism of $A,$  of order~$2$.  Indeed, if $A=A^+\oplus A^-,$
then $y\mapsto y^+-y^-$ is such an automorphism, where $y\in A,$ and
$y=y^++y^-,\, y^+\in A^+,\, y^-\in A^-.$  So any primitive axis
$a\in A$  of type $(\gl,\gd)$ gives rise to an automorphism $\gt_a
\colon A\to A,$ of order~2, given by
\[
y=\ga_y a+y_0+y_{\gl,\gd}\mapsto \ga_y a+ y_0-y_{\gl,\gd},
\]
which in
accordance with the literature we call the {\bf Miyamoto involution
associated with $a$} (even though in standard algebraic terminology,
an algebraic involution really means an anti-automorphism).

Note for any primitive axis $b$ that  $\tau_a (b)\in \langle\langle a,b
\rangle\rangle;$  $\tau_a (b)=b$  iff $ab=0.$ (Indeed, if
$b_{\gl,\gd} = 0$ then $ab = \ga_b a$, contrary to \cite[Lemma
2.4(ii)]{RoSe1}.) For $\gr\in\Aut(A)$ we write $y^\gr$ for $\gr
(y)$, for any $y\in A$.

\begin{lemma}\label{abovename}
Let $a$ be a primitive axis, and $\gr\in\Aut(A).$  Then
\[
\gt_{a^{\gr}}=\gt_a^{\gr}:=\gr^{-1}\gt_a\gr.
\]
\end{lemma}
\begin{proof}
Let $(\gl,\gd)$ be the type of $a.$  Note that for
$\mu\in\{1,0,(\gl,\gd)\},$
\[
x\in A_{\mu}(a^{\gr}) \iff x^{\gr^{-1}}\in A_{\mu}(a).
\]
Let $\gt:=\gt_{a^{\gr}}.$  Then for any $u\in A,$
\begin{gather*}
u^{\gt}=u\iff u\in A_1(a^{\gr})\cup A_0(a^{\gr})\iff u^{\gr^{-1}}\in A_1(a)\cup A_0(a)\\
\iff u^{\gr^{-1}\gt_a}=u^{\gr^{-1}}\iff u^{\gr^{-1}\gt_a\gr}=u.
\end{gather*}
Similarly $u^{\gt}=-u\iff u^{\gr^{-1}\gt_a\gr}=-u.$ Thus
$\gt=\gt_a^{\gr}.$
\end{proof}

Let $\mathcal G(X)$ denote the set of finite products of Miyamoto
involutions of a given set of primitive axes $X$, clearly a group,
called the {\it Miyamoto group} of~$X$. Note that if $X$ is
countable then so is $\mathcal G(X)$.

We write
\[
\Cl(X)=X^{\mathcal{G}(X)}=\{x^{\gr} :\gr\in\mathcal{G}(X)\}.
\]
The set $X$ is {\it closed} if $X=\Cl(X).$

\begin{lemma}\label{abovename1}
\begin{enumerate}\eroman
\item
 $\Cl\Cl(X))=\Cl(X).$ Hence  $\Cl$ is the closure operator in a suitable
topology on the primitive axes of $A$, in which points are closed.

\item
If $\rho \in \mathcal{G}(X)$ then $\langle\langle
{X}^{\rho} \ran\ran = \langle\langle {X} \ran\ran.$
 \end{enumerate}
\end{lemma}
\begin{proof}
(i) By  Lemma~\ref{abovename}, $\calg(X)=\calg(\Cl(X)),$ so
\[
\Cl(X)=X^{\calg(X)}=\Cl(X)^{\calg(X)}=\Cl(X)^{\calg(\Cl(X))}=\Cl(\Cl(X)).
\]

(ii) Since $\gr\in\Aut(\lan\lan {X} \ran\ran),$ we have $X^{\rho} \subseteq
\lan\lan X \ran\ran,$ so $\lan\lan X^{\rho}
\ran\ran \subseteq \lan\lan X \ran\ran.$
The reverse inclusion is obtained by using $\gr^{-1}.$
\end{proof}

A {\it monomial} in  $X$ is defined inductively: Any primitive axis
$x\in X$ is a monomial of {\it length} 1, and if $h$ is a monomial
of length $k$, and $h'$ is a monomial of length $k',$ then $hh'$ is
a monomial of length $k+k'.$

One of our main tools is
the following theorem, a noncommutative version of a result of
\cite{KMS}.

\begin{thm}\label{KMt}
Suppose $A$ is a  PAJ  generated by a set of primitive axes~$X$.
\begin{enumerate}\eroman
\item
(Generalizing \cite[Corollary 1.2, p.~81]{HRS}.) $A$ is spanned by the set~$\Miy(X)$.

\item
If $V\subseteq A$ is
a subspace containing $X$ such that $xv\in V,$ or $vx\in V,$ for all $v\in V$ and $x\in X,$ then $V=A.$

\item
If $X$ is uniform of type $\{\gl,\gd\}$ (resp.~$(\gl,\gd)$ resp.~$\gl$),
then $A$ is spanned by primitive axes of type $\{\gl,\gd\}$ (resp.~$(\gl,\gd)$ resp.~$\gl$).

\item
Let $a\in A$ be a primitive axis of type $(\gl,\gd).$ Then for $\mu\in\{1,0,(\gl,\gd)\},$ we have
\[
A_{\mu}(a)=\sum\{\ff b_{\mu} : b\in \Miy(X)\}.
\]
For example,
$$A_{0}(a)= \sum \{ \ff b_0 : b \in \Miy(X)\}.$$
\end{enumerate}
\end{thm}
\begin{proof} Recall that $\Miy(X)= X^{\calg(X)}$.

(i) By induction on the length of a monomial in the primitive axes $\Miy(X),$
it suffices to show that  $ab$ is in the span of $\Miy(X),$ for
$a,b\in \Miy(X).$ Write
\[
b=\ga_b a+b_{0}+b_{\gl,\gd}.
\]
Then $ab=\ga_b a+ \gl b_{\gl,\gd}.$  But
 $b_{\gl,\gd}= \half(b-b^{\gt_a}).$

(ii)
Let $a\in X$ be of type $(\gl,\gd),$ and $v\in V.$  Then $av=\ga_v a +\gl v_{\gl,\gd},$
and $va=\ga_v a+\gd v_{\gl,\gd}.$  Thus, by hypothesis,
$v_{\gl,\gd}\in V.$
But then $v^{\gt_{a}}=v-2v_{\gl,\gd}\in V.$  Thus $V^{\gt_a}=V.$
As this holds for all $a\in X,$ we see that $\Cl(X)\subseteq V,$ so by (i), $V=A.$

 (iii) $x^{\gr}$ has the same type as $x$, for any automorphism
$\gr$, so apply (i).

(iv) This follows, decomposing $A$ into eigenspaces, since $\Miy(X)$
spans~$A$.
\end{proof}

As a corollary we get

\begin{cor}\label{same type}
Suppose $A$ is generated by $X$ and that $X$ is uniform.
If $X=X_{\{\gl,\gd\}}$ (resp.~$X=X_{\gl}$), then
all the primitive axes in $A$ are of type $(\gl,\gd)$
or $(\gd,\gl)$ (resp.~of type $\gl$ or $1-\gl$).
\end{cor}
\begin{proof}
We start with the case where $X=X_{\{\gl,\gd\}}.$
Set $T:=\{(\gl,\gd),\, (\gd,\gl)\}.$
Let $a$ be a primitive axis in $A.$  Suppose that the type of $a$  is not in $T.$
Then, by \cite[Examples 2.6]{RoSe1}, $ab=0,$ for any primitive axis $b\in A$
whose type is in $T.$  But by Theorem \ref{KMt}(i), $A$ is spanned by
$\Cl(X),$ and of course $\Cl(X)=\Cl(X)_{\{\gl,\gd\}}.$
Hence $aA=0,$ which is impossible since $a^2=a.$

The same argument works when $X=X_{\gl}$ using  \cite[Proposition 2.12(vi)]{RoSe1}
and Theorem \ref{KMt}(i).
\end{proof}

In Theorem \ref{U1}(3) below we get more precise information when
${X=X_{\gl,\gd}.}$

\section{The axial graph}

Let $A$ be a PAJ. Recall that $\chi(A)$ denotes the set of {\it all}
primitive axes of $A.$
We continue to assume that $X$ is a collection of
primitive axes.  As in \cite[\S 1.2.1]{HSS}, define the {\it axial
graph}  on~$\chi(A)$ to be the graph whose vertex set is
$\chi(A),$ and whose edges connect primitive axes $x'$ and $x''$ when
$x'x'' \ne 0$. (Note that since any primitive axis is of Jordan
type, $x'x'' \ne 0$ if and only if $x''x' \ne 0.$)   Given a subset $X\subseteq\chi(A),$
when we discuss the graph of $X,$ we mean the full subgraph induced on $X.$
Many of the following results are
mild generalizations of \cite[\S 6]{HSS}.  By ``component''
we mean ``connected component.''

\begin{lemma}\label{gr11}
Let $A = \lan\langle
X\rangle\rangle.$
\begin{enumerate}\eroman
\item
For any two
primitive axes $a$ and $b$, either $a b^{\gt_a}\ne 0 $ or $b
b^{\gt_a}\ne 0 $; hence, for $a, b \in \chi(A)$  with
$ab \ne 0$, the distance between $b$ and $ b^{\gt_a}$ is at most~2.

\item
For any connected subset $X' \subseteq \chi(A),$ $\Cl(X')$ is also
connected.

\item If  $a,b\in \chi(A)$ are in the same
component, then  $b$ has either the same type of $a$ or the
complementary type, cf.~Definition~\ref{hom}.

\item
If $X$ is connected, then $\chi(A)$ is connected.

\item
Any  component of $\Cl(X)$ is closed.
\end{enumerate}
 \end{lemma}
 \begin{proof}
(i) Suppose  $a$ is of type $\gl$. If $b^{\gt_a}\ne b$ then $b =
\ga_b a + b_0 + b_{\gl},$ with $b_{\gl}\ne 0$,  so $a b^{\gt_a}
=\ga_b a - \gl b_{\gl}\ne 0$. Likewise for $a$  of type $(\gl,\gd).$
\medskip

(ii) Let $b\in X',$ and $\gr\in \mathcal{G}(X'),$ so that $b^{\gr}$
is a typical element in  $\Cl(X').$  Now $\gr$ is a
product $ \gt _{x_1}\dots \gt _{x_m},$ for $x_i \in X'$. We claim
that $b^{\gt _{x_1}}$ is in the same   component as $b$. This is
clear unless $b^{\gt _{x_1}}\ne b$. Then $b x_1\ne 0 ,$ so the claim
follows from (i), taking $x_1 = a.$

By induction on $m$, $b^\gr$ is in the same component as $b.$
\medskip

(iii)
Suppose that $a$ has
type $\gl.$ Of course we may assume that $ab\ne 0,$ so $b$ is of
type $\gl$ or $1-\gl$, by \cite[Proposition 2.12(vi)]{RoSe1}. The
noncommutative case is by \cite[Examples 2.6]{RoSe1}.

(iv)  Suppose that $X$ is connected.  By Theorem \ref{KMt}(i),
$\Cl(X)$ span $A.$  By (ii), $\Cl(X)$ is connected.
Let $a\in\chi(A).$  Then $ax'\ne 0,$ for some $x'\in\Cl(X),$
otherwise $aA=0,$ which is false.  Hence $\chi(A)$ is connected.
\medskip

(v) Let $X'$ be a component of $\Cl(X).$  Then, by (ii), $\Cl(X')$
is connected and is contained in $\Cl(X).$
But $X'\subseteq \Cl(X'),$ and any connected set containing
a component equals the component.
\end{proof}

We say that an element $y\in A$ is {\it annihilating} if $Ay = 0$
(note that $Ay=0$ iff $yA=0,$ because $A$ is spanned by primitive
axes and $ay=0\iff ya=0,$ for a primitive axis $a.$)

\begin{thm}[The axial decomposition] \label{miya1}
 Let $A = \langle\langle
X\rangle\rangle,$ with $\{X_i:i \in I\}$  the components of
$\Miy(X)$, and $A_i = \langle\langle X_i\rangle\rangle$.
\begin{enumerate}\eroman
\item $ A = \sum A_i.$
\item For each
$i\ne j,$ $ A_i A_j=0,$ and  $ A_i \cap A_j$ is annihilating.

\item  If $\sum _i y_i = 0$ for $y_i\in A_i,$ then each
$y_i$ is annihilating in $A.$

\end{enumerate}
\end{thm}
\begin{proof}
(i) By Lemma~\ref{gr11}(v), and Theorem \ref{KMt}(i), $A_i$ is spanned by $X_i.$ Furthermore
${\Cl(X) = \bigcup X_i}$, so
$A = \sum _{x\in\Cl(X)} \ff x = \sum _{i\in I} \sum _{x\in X_i} \ff x
= \sum A_i.$

(ii) By definition $x_ix_j = 0$ for each $x_i\in X_i$, $\ x_j\in
X_j$.
 Since $ A_i $  is spanned by
$X_i,$ we obtain the first part of (ii);
the latter part follows from (i).

(iii) $A_j y_i = 0 $ for $j\ne i.$ Since $\sum y_i = 0$ we have $y_i
= - \sum _{j\ne i} y_j \in \sum _{j\ne i} A_j,$ so $A_i y_i
\subseteq \sum _{j\ne i} A_i A_j = 0,$ by (ii), implying $Ay_i = 0$
by (i).
\end{proof}

One can cull more information.  But first a remark.

\begin{remark}\label{B}
\item
Uniformly  2-generated CPAJ's are described in
\cite[Proposition~4.1]{HRS}, denoted $B(\gl, \varphi)$ there. These
are spanned by primitive axes $a$ and~$b$ of type $\gl$ and the
element  $\gs $ satisfying $ab = \gs +\gl(a+b),$ $a\gs = \gc a,$
$b\gs = \gc b,$ and $\gs ^2 = \gc \gs,$ where $\gc = (1-\gl)\varphi
- \gl.$ It follows  that the algebra  $B(\gl, \varphi)$ is uniquely
determined up to isomorphism by $\gl$ and~$\gc.$
\end{remark}

\begin{lemma}\label{gr1}
Suppose $a,b$ are primitive axes, with $a$ of type $\gl$
and $A: =\langle\langle a,b \rangle\rangle$ commutative. Then  $b^{\gt_a}$ is
a neighbor of $b$ unless   $2w$ is the multiplicative unit of $A$,
where
\[
\textstyle{w: =  \half( b + b^{\gt_a})=\ga_b a + b_0,}
\]
in which case $A$ is the algebra $B(\half,\half)$ of
 \cite{HRS}.
 \end{lemma}
 \begin{proof} We may assume that $b^{\gt_a}\ne b$, so $b_{\gl}\ne 0,$ and $ab \ne 0.$  We are done unless  $b  b^{\gt_a}= 0.$
  Then $b = w + b_\gl,$ so
\begin{equation}
b b^{\gt_a} = (w + b_\gl)(w- b_\gl)  = w^2- b_\gl^2 .
\end{equation}
Hence $w ^2  = b_\gl^2,$ implying
$b = b^2 = 2w^2 + 2w b_\gl$,  so matching components in
$A_1(a)+A_0(a)$ and $A_\gl(a)$ using the fusion rules,
\[
w=2w^2\text{ and }b_\gl =2w b_\gl.
\]
Now  $A_0^2(a) \subseteq A_0(a)$ by
\cite[Corollary 2.11]{RoSe1}, so, matching components in
 $$\ga_b a+b_{0}= 2(\ga_b ^2 a+b_0^2 ) $$ shows that
$$2b_0 ^2 =b_0  ,\qquad 2 \ga_b ^2=
 \ga_b ,$$ so $ \ga_b \in \{ 0,\half\}.$

If $\ga_b = 0$ then $b = b_0 + b_\gl.$ Then
\[
0 = bb^{\gt_a} = b(b_0 -b_\gl),
\]
implying
$b b_0 = b b_\gl.$ Then $b = b^2 = b( b_0+b_\gl) = 2bb_0,$ so $bb_0 = \half b = b b_\gl.$
Hence $bA = b( \ff a +\ff b_0 + \ff b_\gl) \subseteq\ff b a + \ff b b_\gl = \ff b_0 + \ff b_\gl,$
because $ba\in\ff b_{\gl},$ and $bb_0=bb_{\gl}\in\ff b\subseteq \ff b_0 + \ff b_\gl.$  Hence the
two-dimensional space $\ff b_0 + \ff b_\gl$ contains the two
eigenvectors of $b$ having nonzero eigenvalues, as well the
$0$-eigenvector  $b_0 - b_\gl,$ a contradiction.

Thus $\ga_b = \half$. Hence $2w$ is the unit element of $A$, since
\[
(2w)a = 2\ga_b a = a,\qquad (2w)w = 2 w^2 = w, \qquad (2w) b_{\gl}
= b_{\gl}.
\]

To conclude, let $b$ have type $\gl'$. We need to show that $\gl =
\gl' = \half.$ Recall from \cite[Proposition 2.10(1)]{RoSe1} that
\begin{gather*}
\textstyle{\gs : = ab - \gl' a - \gl b = \half a+\gl b_{\gl}- \gl' a- \frac{\gl}{2}a -\gl b_0-\gl b_{\gl}}\\
\textstyle{=(\half-\gl'-\frac{\gl}{2})a-\gl b_0.}
\end{gather*}
satisfies $\frac{\gs} \gc
 = \one = 2w = 2 b_0+a,$ where $\gc = \half(1-\gl) - \gl'.$ Matching coefficients of $b_0$ shows $-\frac
 \gl \gc
 = 2,$ so $$ -\gl = 2 \gc = (1-\gl) - 2\gl',$$ implying $\gl' =
 \half$ and thus also $\gl =
 \half,$ by \cite[Proposition 2.12(iv)]{RoSe1}.
\end{proof}

We would like the sum in Theorem~\ref{miya1}(i) to be a direct
product, but this need not be the case:

\begin{example}\label{non}$ $
\begin{enumerate}\eroman
\item
The algebra  $B(\gl,\gvp)$ (see Remark \ref{B}) is $2$-dimensional precisely when $ab=0$ (in which case
$\gs=-\gl(a+b),$ and $\gvp=0$)  or
$\gl\in\{-1,\half\}$ and $\gs=0.$

Let $A=\lan\lan a,b'\ran\ran$ be a CPAJ such that $a,b'$ are
primitive axes, $ab'\ne 0,$ and $a\ne b'.$  If $\dim A=2,$ then by
\cite[Theorem A and Proposition~C]{RoSe1}, $A\cong B(\gl,\gvp),$ and
the type of $a$ equals the type of $b'.$

Otherwise, by \cite[Theorem 2.16]{RoSe3}, $\dim A=3,$ and by
\cite[Theorem B]{RoSe1}, $A$ is one of the CPAJ's of \cite{HRS}.
Suppose $a$ is of type~$\gl.$ By \cite[Proposition
2.12(iii)\&(vi)]{RoSe1} (see also Remark \ref{gen} below), if $\gl\ne -1,$ then one can find a
primitive axis $b$ of type $\gl$, with $A = \langle\langle
a,b\rangle\rangle.$  When $a$ is of type~$-1,$ then $b'$ is of type~
$2$ and one can find a primitive axis $b$ of type $2,$ with $A =
\langle\langle b,b'\rangle\rangle.$

Suppose $\dim B(\gl,\gvp)=3.$
When $\gc \ne 0,$ $\frac 1{\gc} \gs$ is the unit element of $B(\gl,
\varphi)$, which we denote as $\one =\one_{B(\gl, \varphi)}.$ In
this case $\one - a $ is a primitive axis of type $1-\gl$ in $B(\gl,
\varphi),$ so  when $\gl\ne 2,$ $B(\gl, \varphi)= B(1-\gl, \varphi') $ for some
$\varphi'$, leading to a slight ambiguity.

The remaining case is when $\dim(B(\gl,\gvp))=3,$ and $\gc = 0.$ In this case
 $\gl\in \{\half, -1\}$, which does occur, as explained in
\cite[Theorem~1.1]{HRS}. Then $\gs A =0,$ i.e., $\gs$ is
annihilating.

\item
(An example of a CPAJ $\bar A =  \langle\langle \bar X \rangle\rangle $
having  two subalgebras generated by disjoint components of $\Cl(\bar X)$, with nontrivial intersection.)
 Suppose $B(\gl',\varphi')= \langle\langle a',b'\rangle\rangle  $ and $ B(\gl'',
\varphi'')= \langle\langle a'',b''\rangle\rangle $ are 3-dimensional
CPAJ's which have $\gc' = \gc'' = 0$, i.e, the  respective elements
$\gs'$ and $\gs''$ have square zero. Let $A =B(\gl', \varphi')
\times B(\gl'', \varphi'')$ and $\bar A = A/I,$ where $I$ is the
ideal generated by $\gs'- \gs''.$ Let $\bar X$ be the set of images
$\{\bar a',\bar b', \bar a'',\bar b''\}$, so $\Miy(\bar X)$ has
disjoint
 components  $\Miy(\bar a',\bar b' )$ and
$\Miy(\bar a'',\bar b'')$.  $\bar A =  \langle\langle \bar X
\rangle\rangle$ is a CPAJ
 in which  $\bar \gs'= \bar\gs'' \in
\overline{B(\gl', \varphi')} \cap \overline{B(\gl'', \varphi'')}.$
(As noted in (i), one could have different types $\gl'=\half$ and
$\gl''=-1$.)

\item
The noncommutative situation will be treated in depth in
\S\ref{nonPAJ}, but now let us quickly recall from \cite[Examples
2.6]{RoSe1} that there are two fundamental cases:

In Example 2.6(1), $\dim A =2$ and all primitive axes are of the same type.

In Example 2.6(2), $\dim A = 3$ and $A$  cannot be generated by primitive axes
of the same type, although $A$  is uniformly generated. In this
situation we can construct the same kind of example as in (ii), in
which the algebras of disjoint components intersect at a nonzero
annihilating ideal.
\end{enumerate}
   \end{example}

\begin{prop}\label{A032}
Let $\{X_i : i \in I\}$  be the
components of $\Miy(X)$. Then for any $j$, $N_j:=\langle\langle
X_j\rangle\rangle \cap \langle\langle \Miy(X)\setminus
X_j\rangle\rangle )$ is an annihilating ideal of $A:=\langle\langle
X\rangle\rangle $, and
$$ A/N_j \cong \langle\langle
X_j\rangle\rangle /N_j \ \times \ \langle\langle X\setminus
X_j\rangle\rangle/N_j.$$
\end{prop}
\begin{proof} $N_j \triangleleft A$ is annihilating by Theorem~\ref{miya1}(iii), and the
next assertion follows at once.
\end{proof}

\subsection{The strong axial graph}$ $
\medskip

Each  component of the axial graph can be refined further.

\begin{Def}
Define the {\it strong axial
graph}  of~$\chi(A)$ to be the graph whose vertex set is again
 $\chi(A)$, and whose edges connect axes $x'$ and $x''$
if $x'=x''$ or $\dim \langle\langle x',x'' \rangle\rangle = 3.$
We then say that $x'$ is {\it strongly connected} to $x''.$

A subset  $X'\subseteq \chi(A)$ is strongly connected if it is connected in
the strong axial graph.

An axis $a$ is {\it strongly connected} to a subset $
X'\subseteq\chi(A)$ if $a$ is strongly connected to some primitive
axis of~$X'$.

Suppose that $A$ is $n$-generated.  For any generating set
$Y\subseteq \chi(A)$ of $A$ of size $n$ we let $Y_1, Y_2,\dots, Y_k$
be the strong connected components of $Y$ (as a full subgraph of the
strong axial graph),  with $|Y_1|\ge |Y_2|\ge\dots\ge |Y_k|\ge 1.$
If~$Y'$ and $Y'_1, Y'_2,\dots Y'_{k'}$ are as above, we write
$Y>>Y'$ if for the first $i$ such that $|Y_i|\ne |Y'_i|,$ we have
$|Y_i| > |Y'_i|.$  $Y$ is a {\it good generating set of size $n,$}
if it is maximal in the above order relation $>>.$
\end{Def}

\begin{thm}\label{strongcon}
Let $ A =\langle\langle X \rangle\rangle$ where $X$
is a good generating set of $A$ of size $n.$  Let $X_1,\dots,X_k$
be the strong connected components of $X.$
Then $ A = \sum
\langle\langle X_j \rangle\rangle,$  so $\dim A \le \sum \dim
\langle\langle X_j \rangle\rangle.$
\end{thm}
\begin{proof}
We may assume that  $|X_1|\ge |X_2|\ge\dots\ge|X_k|.$
We claim that if $a\in \Cl(X_j)$ and $b\in \Cl(X_k)$ for $j< k$,
then $ab \in \ff a + \ff b$, i.e., $a$ and $b$ are not strongly
connected. We prove the claim by showing that the negation, that $a$
and $b$ are strongly connected, leads to a contradiction. Write $a =
{x}^{\gr}$ for $x\in X_j$ and $\gr \in \calg (X_j).$
 We may replace
$X_j$ by $X_j^\gr,$  by Lemma~\ref{abovename1}(ii) and still obtain
a generating set of $A$ of size $n.$ Now write $b = \gr'(x')$ for
$x' \in X_k$ and $\gr' \in \calg(X_k).$ In view of
Lemma~\ref{abovename1}(ii), we may replace $X_k$ by~
$X_k^{{\gr'}^{-1}}.$  After these replacements we obtain a
generating set $X'$ of $A$ of size~$n,$ with $X'>> X.$ Indeed
$X_j^{\gr}\cup\{b\}$ is strongly connected. This contradicts the
maximality of $X.$

The claim implies that $ \sum \langle\langle X_j \rangle\rangle$ is
multiplicatively closed (by means of~Theorem~\ref{KMt}(i)) and
contains $X$, so
 we see that $ \sum \langle\langle X_j
\rangle\rangle =A .$
\end{proof}

\section{The structure of noncommutative PAJ's}\label{dePAJ}$ $

In this section we strengthen Theorem~\ref{miya1}, showing that for
any noncommutative component $X_i$, $A_i = \langle\langle
X_i\rangle\rangle$ is a factor in a direct product.

The following lemma is very useful.

\begin{lemma}\label{e}
Suppose $A=\lan\lan X\ran\ran$ and let $e\in A$ be an idempotent.
Then
\begin{enumerate}
\item
If $a\in\ff e +A_0(e)+A_{\gl,\gd}(e),$ (we allow $\gl=\gd$) for all
$a\in\Cl(X),$ then $A=\ff e+A_0(e)+A_{\gl,\gd}(e).$

\item
Assume that (1) holds and for any $a, b\in\Cl(X)$ we have
\begin{equation}\label{eqe}
a_0b_0, a_{\gl,\gd}b_{\gl,\gd}\in \ff e+A_0(e),\qquad a_0b_{\gl,\gd}, b_{\gl,\gd}a_0\in A_{\gl,\gd}(e),
\end{equation}
where $a=\ga_a e+a_0+a_{\gl,\gd}$ and $b=\ga_b e+b_0+b_{\gl,\gd},$ are the decompositions
of $a, b$ with respect to $e.$

If $L_eR_e=R_eL_e,$  then $e$ is a primitive axis in $A.$
\end{enumerate}
\end{lemma}
\begin{proof}
Part (1) holds as by Theorem \ref{KMt}(i), $A$ is spanned by
$\Cl(X).$ For part~(2) note that as $A$ is spanned by $\Cl(X),$
using (1) we see that any element in $A_{\gl,\gd}(e)$ can be written
as $\sum_{x\in\Cl(X)}\ga_x x_{\gl,\gd},$ where this is a finite sum
and $\ga_x\in \ff.$  Similarly the same holds for any element in
$A_0(e).$  So to check the fusion rules it suffices to check that
equation \eqref{eqe} holds.  Thus if $L_eR_e=R_eL_e,$ then
$e$ is a primitive axis in $A.$
\end{proof}

\subsection{Uniformly generated noncommutative  PAJ's} \label{nonPAJ}$ $
\medskip

First we classify all non-commutative PAJ's, starting from
\cite[Example~2.6]{RoSe1}.

\begin{remark}\label{ugen}
By \cite[Theorem 2]{RoSe2} every noncommutative 2-generated
PAJ is either {\cite[Example~2.6(1)]{RoSe1}}, defined via
multiplication
\[
ab = \gd a + \gl b; \qquad ba = \gd b + \gl a,
\]
with $\gl\ne\gd,$ and $\gl + \gd = 1$, or  the ``exceptional'' axial algebra,
\cite[Example~2.6(2)]{RoSe1}, to be described presently in greater
detail. These examples show that {\it every connected set of
noncommutative primitive axes is uniform.}
\end{remark}

Let us elaborate.

\begin{example}[{The ``exceptional'' axial algebra, \cite[Example~2.6(2)]{RoSe1}}]\label{try1}

Let  $B:=A_{\operatorname{exc},3}(\{a,b\},\gl)$ denote the 3-dimensional
algebra spanned by primitive  axes $a,b$ and an element $y,$ where
\[
ab = ay= yb =\gl y ,\quad ba = ya =  \gd y=by,\quad y^2=0,
\]
for $\gd = 1-\gl, \gd\ne\gl.$

Note the decomposition $b = (b-y) +y,$ so $b_0 = b-y$.
 Hence  $b_0^2
  = (b-y)^2 = b - (\gl+\gd)y = b_0.$

In fact, it was shown in \cite[Example~2.6(2)]{RoSe1} that all
 idempotents of~$B$ are either  $a+b-y$, or the primitive axis  $a + \mu y,$ of type
$(\gl,\gd)$, or the primitive axis $b + \mu y$ of type $(\gd,\gl)$, where
$\mu\in \ff$. Note that
\[
\textstyle{b^{\gt_a} = b-2y, a^{\gt_b}=a-2y,\text{ so }y = \half(b-b^{\gt_a})=\half(a-a^{\gt_b}).}
\]
The  idempotent $a+b-y$ is the multiplicative unit   since
$$(a+b-y)y = \gl y + \gd y +y^2 = y+0 =y,$$ $$(a+b-y)b = \gl y + b -\gl
y = b ,\qquad (a+b-y)a = a+ \gd y  -\gd y = a  .$$ While $B$ is
uniformly generated, its subalgebra $\ff a + \ff y$ only has primitive axes of
type $(\gl,\gd)$, whereas $\ff b + \ff y$  only has primitive axes of type
$(\gd,\gl)$.
   \end{example}

\subsubsection{The noncommutative uniformly generated PAJ\, $\uu_E(\gl)$}

\begin{Def}\label{uuI}
$\uu_E (\gl)$ ($\gl \ne 0,1$)  is the algebra  having as a basis the set of idempotents
$E:= \{ e_i: i \in I\}$, with multiplication given by $e_ie_j = \gd e_i + \gl e_j,$ for each $i,j\in I$, where $\gd = 1 - \gl.$
As we shall see in Lemma \ref{p1}
below, $e_i$ is a primitive axis, for all $i\in I,$ so $U_E(\gl)$ is a PAJ.

When $|E| = n,$ we write $\uu_n(\gl)$ for $\uu_E (\gl)$.
\end{Def}

\begin{lemma}\label{p1}
Let $A = \uu_E(\gl)$. For any $a\in E$,
\[
A_{\gl,\gd}(a)= \sum_{i,j\in I} \ff (e_i - e_j)=\sum _{i\in I}\ff (e_i - a),
\]
so $A_{\gl,\gd}(a)$ is a square-zero ideal.
$A = A_{\gl,\gd}(a) \oplus \ff a,$ and $A/A_{\gl,\gd}(a) \cong \ff a
\cong \ff,$ implying $A$ is local and $A_{0}(a)= 0.$ In particular,  if $Az = 0$ for $z \in A,$ then $z
= 0$.

The nonzero idempotents of $\uu_E(\gl)$ all are $\{ a +z: z \in
A_{\gl,\gd}(a)\}$, and these are all primitive axes. In particular
the elements of $E$ are primitive axes.  If $e,f\in \uu_E(\gl)$
are idempotents, then $ef=\gd e+\gl f$ and $e-f\in A_{\gl,\gd}(e)=A_{\gl,\gd}(a).$
\end{lemma}
\begin{proof}
 Take $a = e_1$. Then $e_j^2 = e_j=\gd e_j + \gl e_j,$ and for each
$i\ne 1$,
\begin{align*}
a (x_i -a) &= \gd a+\gl x_i-(\gd a+\gl a) = \gl (x_i-a),\\
(x_i-a)a&=\gd x_i+\gl a-(\gd a+\gl a)=\gd(x_i-a),
\end{align*}
proving $\sum _{i\in I} \ff (x_i - a)\subseteq A_{\gl,\gd}(a),$
 and  has codimension $1$,
so $A_{\gl,\gd}(a) = \sum _i \ff (x_i - a)$ and $A = \ff a \oplus
A_{\gl,\gd}(a).$ Furthermore
\[
(x_i -a)(x_j-a) = x_ix_j - (x_ia+
ax_j) +a = \gd x_i + \gl x_j - (\gd x_i +  \gl x_j)  = 0,
\]
showing that $A_{\gl,\gd}(a)$ is a square-zero ideal and
yielding the fusion rules.

Next, since  $\gd^2+\gl=\gl^2+\gd:$
\begin{align*}
&L_aR_a(e_i)=a(e_ia)=a(\gd e_i+\gl a)=\gd(\gd a+\gl e_i)+\gl a\\
&=(\gd^2+\gl)a+\gd\gl e_i=(\gl^2+\gd)a+\gl\gd e_i=\gd a+\gl(\gd e_i+\gl a)\\
&=(\gd a+\gl e_i)a=(ae_i)a=R_aL_a(e_i).
\end{align*}
As this holds for all $i,$ we have $L_aR_a=R_aL_a.$
Thus $a$ is a primitive axis.

To see when $y =\nu  a + z$ is idempotent, where $z\in A_{\gl,\gd}(a),$  compute
\[
y^2 = \nu^2 a + \nu(az + za) = \nu^2 a + \nu(\gl z + \gd z) =  \nu^2 a +\nu z.
\]
So $\nu = 1,$ and then any $z$ works.

Take the idempotent $y=a+z$. Then
for any $z'$ in $A_{\gl,\gd} (a)$,
\[
yz'=(a+z)z' = az' = \gl z',\quad z'(a+z)=z'a=\gd z'.
\]
Thus $A_{\gl,\gd}(y)=A_{\gl,\gd}(a).$  Also
\[
(a+z)(a+z')=a+\gl z'+\gd z=\gd (a+z)+\gl (a+z').
\]
Thus, as above $L_yR_y=R_yL_y,$ and $y$ is a primitive axis.
It is easy to check that $e-f\in A_{\gl,\gd}(e).$
\end{proof}

We want the converse.

 \begin{thm}\label{U1}
Let $A$ be a PAJ generated by a set of primitive axes $X$ of type $(\gl, \gd)$ for $\gl \ne \gd.$
Then $A$ is spanned by $X.$ Furthermore
\begin{enumerate}
\item
If $X$ is connected then $A$ has a basis $\calb\subseteq X$ such that
$A=\uu_{\calb}(\gl).$  Moreover,
if $a, b\in \Cl(X),$ then $ab\ne 0.$

\item
Let $\calb\subseteq X$ be a basis of $A,$ and let $\{\calb_i\mid i\in I\}$
be the connected components of $\calb.$  For each $i\in I,$ let
$X_i$ be the connected component of $\Cl(X)$ containg $\calb_i,$
and let $A_i:=\lan\lan X_i\ran\ran.$
Then $\calb_i$ is a basis of $A_i,$ and $A_i=\uu_{\calb_i}(\gl).$
Furthermore
\[
\textstyle{A=\prod_{i\in I} A_i=\prod_{i\in I}\uu_{\calb_i}(\gl).}
\]

\item
Let $a\in A$ be a nonzero idempotent.  Then $a$ is a primitive axis
in $A$ iff $a\in A_i,$ for some $i.$  In particular, all axes in
$A$ are of type $(\gl,\gd).$
\end{enumerate}
\end{thm}
\begin{proof}
Note that by Remark \ref{ugen} and by Example \ref{try1},
$ab \in\{ 0, \gd a +\gl b\},$ for all  $a,b\in X,$
so $A$ is spanned by $X,$ by induction on the length of a monomial in $X.$
\medskip

(1)
Let $a, b\in \Cl(X).$  By Lemma \ref{gr11}(ii), $\Cl(X)$ is connected.
Suppose  that $ab = 0.$  We may assume that
there exists $c\in \Cl(X),$ with $ac\ne 0\ne bc.$
Thus $ac = \gd a + \gl c,$ and
$c b= \gd c + \gl b.$ Suppose $\ga a+\gb b+\gc c=0,\ \ga,\gb,\gc\in\ff.$
Multiplying by $a$ we get $\ga a+\gc ac=0.$  Since $a,c$ are linearily independent,
we see that $\ga=0=\gc,$ so $\gb=0,$ and $a,b,c$ are linearly independent.

Now  $c-a\in A_{\gl,\gd}(a),$ by Lemma \ref{p1}.
Since $b\in A_0(a),$  the fusion rules imply that $cb=(c-a)b\in A_{\gl,\gd}(a).$
But $cb=\gd c+\gl b,$ so $\gd ac=a(cb)=\gl cb.$  This is impossible
since $a,b,c$ are linearly independent.

Hence $X$ contains a basis $\calb$ such that $A =\uu_{\calb}(X).$

\medskip
(2)
Let $x_j\in X_j,$ and write
\[
\textstyle{x_j=\sum_{x\in \calb_j}\ga_x x+\sum_{x\in \calb\sminus \calb_j}\ga_x x.}
\]
Suppose that $\ga_{x'}\ne 0,$ for some $x'\in\calb\sminus \calb_j.$
Then
\[
\textstyle{0\ne v:=\sum_{x\in \calb\sminus \calb_j}\ga_x x\in  A_j.}
\]
By (1), $A_j:=\uu_{\calb'}(\gl),$ for some basis $\calb'$ of $A_j,$
and $v\in A_0(b),$ for each $b\in\calb_j.$
But by Lemma \ref{p1}, $A_0(b)=0,$ a contradiction.

Hence $X_j$ is contained in the span of $\calb_j,$ and so $\calb_j$ is a basis of $A_j,$
and $A_j=\uu_{\calb_j}(\gl).$ Now $A=\sum_{i\in I}A_i,$ and since, by Lemma \ref{p1},
$A_i$ contains no annihilating elements, $A=\prod_{i\in I}A_i,$ by Theorem \ref{miya1}.
\medskip

(3)
Clearly $a=\sum_{i\in I}a_i,$ where $a_i\in A_i$
is an idempotent and $a_ia_j=0,$ for $i\ne j.$

Let $i\in I,$ with $a_i\ne 0.$  Then $aa_i=a_i,$ so $a$ is primitive
iff $a=a_i.$  By~Lemma~\ref{p1}, all idempotents in $A_i$ are
primitive axes in $A_i,$ and hence also in $A.$
\end{proof}

\begin{remark}
$\uu_E (\gl)$ is  commutative if and only if $\gl = \half.$ The
assertion of~Theorem~\ref{U1} (and its proof) holds true in this
case as well.

The case $\gl  = 1$ has been excluded, since $\uu_2(1)$  would
satisfy $ab =b$, so the axis $a$ is not primitive. (Likewise for
$\gl  = 0$ and $\uu_2(0).$ )
\end{remark}

\subsubsection{The prototypical uniformly generated noncommutative PAJ}$ $
\medskip

We complete the classification of noncommutative PAJ's. This  relies
on~Theorem~\ref{U1} and also on the exceptional algebra
$A_{\operatorname{exc,3}}(\{a,b\},\gl)$.

We  start by classifing all noncommutative, uniformly generated  PAJ's $A =
\langle\langle X\rangle\rangle$, of type $\{\gl,\gd\}$ for $\gl \ne
\gd$.  Of course each primitive axis in $\Cl(X)$ is of type $(\gl,\gd)$ or
$(\gd,\gl)$. But we can  say more.

\begin{thm}\label{almostcomm}
Suppose that $A$ is a uniformly generated PAJ, generated by a
set $X$ with $X' = X_{\gl,\gd}$ and $X'' = X_{\gd,\gl},$ where $\gl
+ \gd = 1$ and $\gl \ne \gd.$  Set $A':=\langle\langle X'\rangle\rangle$
and $A'':=\langle\langle X''\rangle\rangle.$
Let $Z =\{ \sum \ff x'x'': x'\in X',\ x''\in X''\}. $ Then
 the rest of the
multiplication table of $A$ is described by the following products:
\begin{enumerate}\eroman
\item
Let $a, a'\in A$ be two primitive axes of type $(\gl,\gd)$ and $b\in
A$ be an axis of type $(\gd,\gl).$  Then  $ab, ba\in
A_{\gl,\gd}(a')$  if $aa'\ne 0,$ while $ab, ba\in A_0(a')$ if
$aa'=0.$  In particular if $x'\in A$ is an axis of type $(\gl,\gd),$
with $a'x'\ne 0,$ then $A_0(a')\cap Z=A_0(x')\cap Z$ and
$A_{\gl,\gd}(a')\cap Z=A_{\gl,\gd}(x')\cap Z.$

\item
$Z$ is an ideal in $A,$ so
$A = A' + A'' + Z.$
Furthermore $\widebar{A}:=A/Z\cong \widebar{A'}\times\widebar{A''}.$

\item
$\dim A \le \dim \langle\langle X'\rangle\rangle +\dim \langle\langle X''\rangle\rangle +\dim(Z).$
In particular, if $X$ is finite, then $\dim A$ is finite.

\item
$Z^2 = 0.$

\item
Let $a\in A$ be a primitive axis of type $(\gl,\gd).$
Then
\begin{gather}\label{w'}
\text{ if }w\in A_{\gl,\gd}(a), \text{ then }w=v'+z',\\\notag
\text{ with }v'\in(\ff a+A')\cap A_{\gl,\gd}(a)\text{ and }z'\in Z\cap A_{\gl,\gd}(a).
\end{gather}
In particular,  if  $a\in X',$  then
\[
A_{\gl,\gd}(a) =A'_{\gl,\gd}(a)+Z\cap A_{\gl,\gd}(a).
\]
Hence if $a_1, a_2\in X'$ are in the same connected component of $X',$
then $A_{\gl,\gd}(a_1) =A_{\gl,\gd}(a_2).$  Furthermore $A_{\gl,\gd}(a)^2=0.$

\item
Let $c$ be a primitive axis in $A'.$  Then $c$ is a primitive axis in $A.$
The same statement holds with $A''$ in place of $A'.$

\item
Suppose $a\in X'$ and $z\in Z.$ If $a+z$ is an idempotent in $A$
(i.e.,  $z\in A_{\gl,\gd}(a)$), then $a+z$ is a primitive axis of
the same type as $a.$  In particular, if $X'$ is connected, then
$a+z$ is a primitive axis for all $a\in X'$ and $z\in Z.$  The same
statement holds with $X''$ in place of~$X'.$

\item
Suppose $X'$ is the set of all primitive axes of $A$ of type $(\gl,\gd).$
Taking one
 primitive axis $a_j'$ for each connected component of $X'$
 we have $Z\subseteq \sum_{j\in J}A'_{\gl,\gd}(a'_j).$
The same assertion holds with $X''$ in place
of~$X'$ (and $(\gd,\gl)$ in place of $(\gl,\gd)$).

\item
If $X'$ is connected, then $A_{\gl,\gd}(a)=\sum \ff_{x'\in X'} (x'-a)+Z,$ for all $a\in X',$
so $A_{\gl,\gd}(a)$ is an ideal of $A,$ and $A/A_{\gl,\gd}(a) \cong \ff\times \widebar{A''}.$

 Suppose $X=\chi(A).$ Taking one
 primitive axis $a_j'$ for each connected component of $X',$ and  one
 primitive axis $a_j''$ for each connected component of $X''$, let
\begin{equation}\label{I}
\textstyle{I:=\sum _j  A'_{\gl,\gd}(a'_j) +\sum _j   A''_{\gl,\gd}(a_j'').}
\end{equation}
Then $I$  is an ideal in $A,$ and $A/I$ is
 a direct product of fields.

\item
$A$ has no nonzero annihilating element. Hence $A$ is a direct
product of the $A_i$'s where $A_i$ is as in Theorem \ref{miya1}.
 \end{enumerate}
\end{thm}

\begin{proof}
(i)
We may assume that $z:=ab\ne 0.$  Then
$\lan\lan a,b\ran\ran\cong A_{\operatorname{exc,3}}(\{a,b\},\gl).$
Let $y:=\frac{1}{\gl}z.$
Note that by Example \ref{try1}, $a-2y$ is a primitive axis
in $A$ of type $(\gl,\gd).$

Suppose first that $a'a\ne 0.$  We claim that $a'(a-2y) \ne 0.$
Indeed, by Example \ref{try1}, $a(a-2y)\ne 0,$ so, by Theorem \ref{U1}(1), $a'(a-2y) \ne 0.$

By Theorem \ref{U1} and
Lemma \ref{p1},
$2y=a-(a-2y) \in A_{\gl,\gd}(a'),$ so $a'z = \gl z,$ and $za'=\gd z.$
Note that $ab=\frac{\gl}{\gd}ba$ as in Example \ref{try1}, so $ba\in A_{\gl,\gd}(a').$

 Suppose next that $a'a = 0$ then $a' \in A_{0}(a)$ so, since $y\in A_{\gl,\gd}(a),$ by the fusion rules $a'y \in
 A_{\gl,\gd}(a),$ but if $ a'(a-2y) \ne 0$ then
\[
a'(-2y) = a'(a-2y)= \gl (a-2y) + \gd a'= \gl (-2y) +\gl a +\gd a',
\]
so
\[
\gl^2(-2y)+\gl^2 a+\gl\gd a'=\gl a'(-2y)=a(a'(-2y))=\gl^2 (-2y)+\gl a.
\]
Thus $\gl a+\gd a'=a,$ so $a=a',$ a contradiction. Hence $0 =
a'(a-2y)=a'(-2y),$ so $z\in A_0(a'),$ and then also $ba\in A_0(a').$
This proves the first part of~(i).

To show that  $A_{\mu}(a')\cap Z=A_{\mu}(x')\cap Z,$ for $\mu\in\{0, (\gl,\gd)\},$
it suffices to show that $b'b''\in A_{\mu}(a')\iff b'b''\in A_{\mu}(x'),$ for $b'\in X'$ and $b''\in X''.$
But since $a'x'\ne 0,$ Theorem \ref{U1}(1) implies that $a'b'\ne 0\iff x'b'\ne 0.$
Hence the second part of (i) follows the first.
\medskip

(ii) Since $\Cl(X)$ consists of primitive axes of type $(\gl,\gd)$ or $(\gd,\gl),$
and since, by Theorem \ref{KMt}(i), $A$ is spanned by $\Cl(X),$ the fact that
$Z$ is an ideal follows from (i).

By (i), $a'Z, Za', a'' Z, Za''\subseteq Z,$ for all $a'\in X'$ and $a''\in X''.$
By Theorem~\ref{U1}, $\langle\langle S\rangle\rangle$ is spanned by $S,$ for $S\in\{X', X''\},$
so $\langle\langle X'\rangle\rangle \langle\langle X''\rangle\rangle\subseteq Z.$ Since
$Z$ is an ideal of $A,$ it follows that
$A' + A'' + Z,$
is a subalgebra of $A$ that contains $X,$ so it equals $A.$

Next, $\widebar{A'}\,\widebar{A''}=\bar 0,$ since $A'A''\subseteq Z.$
Also if $\bar w\in \widebar{A'}\cap
\widebar{A''},$ then $\bar w$ is annihilating in $\widebar{A'},$ so
by Theorem \ref{U1}, $\bar w=0.$
\medskip

(iii)  This follows from (ii).
\medskip

(iv)  It suffices to show that $(ab)(cd)=0,$ for $a,c\in X'$ and
$b,d\in X''.$ If $ab=0,$ this is obvious.  Otherwise  $a$ and $b$
generate $A_{\operatorname{exc,3}}(\{a,b\},\gl).$ Setting
$y:=\frac{1}{\gl}ab,$ we saw that $a, a-2y$ are two primitive axes
of type $(\gl,\gd)$ in $A.$ If $cd\in A_{\gl,\gd}(a)\cap
A_{\gl,\gd}(a-2y),$ or $cd\in A_0(a)\cap A_0(a-2y),$ then
$(2y)(cd)=(a-(a-2y))(cd)=0,$ so $(ab)(cd)=0.$  Assume that
$a(cd)=\gl (cd),$ and $(a-2y)(cd)=0.$  Then $(2y)(cd) =\gl (cd),$ so
$y(cd)\in A_{\gl,\gd}(a).$ However, since $y, cd\in A_{\gl,\gd}(a),$
the fusion rules imply that $y(cd)\in \ff a+A_0(a).$  So again
$y(cd)=0,$ and hence $(ab)(cd)=0.$ A similar argument works if
$a(cd)=0$ and $(a-2y)(cd)=\gl (cd).$
\medskip

(v)  We first claim that
\[
\text{if }w''\in A'',\text{ then }w''\in A_0(a)+z,\text{ with }z\in Z\cap A_{\gl,\gd}(a).
\]
Indeed let $w''=\gvp_a(w'')a+w''_0+w''_{\gl,\gd},$ be the decomposition of
$w''$ with respect to $a.$  Then $aw''\in A_{\gl,\gd}(a),$ by (i).  Hence
\[
A_{\gl,\gd}(a)\ni aw''=\gvp_a(w'')a+\gl w''_{\gl,\gd},
\]
so $\gvp_a(w'')=0.$
But now $Z \ni aw''=\gl w''_{\gl,\gd},$  so $w''_{\gl,\gd}\in Z.$

Note that by (i), $z\in A_0(a)+A_{\gl,\gd}(a),$  for all $z\in Z.$

Let now $w=w'+w''+z\in A,$ with $w'\in A',\ w''\in A''$ and $z\in Z.$
Assume $w\in A_{\gl,\gd}(a)$ and let $w'=\gvp_a(w')a+w'_0+w'_{\gl,\gd},$
$w''=w''_0+w''_{\gl,\gd},$ and $z=z_0+z_{\gl,\gd}$ be the decompositions with respect to $a.$
Matching the $(\gl,\gd)$ component we get
\[
w=w'_{\gl,\gd}+w''_{\gl,\gd}+z_{\gl,\gd}.
\]
Note that by Theorem \ref{U1}, $\lan\lan X'\cup\{a\}\ran\ran$ is
spanned by $X'\cup\{a\},$ so $\ff a+A'$ is a subalgebra of $A.$  Thus, $w'_{\gl,\gd}\in \ff a+A'.$
We saw already that $w''_{\gl,\gd}\in Z,$ so equation \eqref{w'} holds.

Next note that by Theorem \ref{U1} and Lemma \ref{p1},
$A'_{\gl,\gd}(a_1)=A'_{\gl,\gd}(a_2),$ and $a_1a_2\ne 0,$ so  by (i), $A_{\gl,\gd}(a_1)\cap Z= A_{\gl,\gd}(a_2)\cap Z.$

Finally, by Lemma \ref{p1}, $A'_{\gl,\gd}(a)^2=0,$ and by (iii) $Z^2=0.$  Also, by (i),
$A'_{\gl,\gd}(a)\big(Z\cap A_{\gl,\gd}(a)\big)=0,$ so $A_{\gl,\gd}(a)^2=0.$
\medskip

(vi) Since $c$ is a primitive axis in $A',$  Theorem \ref{U1}(3)
implies that $c=\sum_{i=1}^k\ga_i x'_i$ with $0\ne x'_i\in X',$ and
$x'_ix'_j\ne 0,$ for all $i,j,$ and $\{x'_1,\dots, x'_k\}$ are
linearly independent. Let $\calx':=\Cl(X)_{\gl,\gd}.$  Since
$\{x'_1,\dots,x'_k\}$ are axes in $\lan\lan \calx'\ran\ran,$ and
since $x'_ix'_j\ne 0,$ for all $i,j,$ by~Theorem~\ref{U1} there
exists a component $Y'\subseteq \calx'$ such that
$\{x'_1,\dots,x'_k\}\subseteq Y'.$ But  $\lan\lan Y'\ran\ran\cong
U_B(\gl),$ by~Theorem~\ref{U1}, and $c\in \lan\lan Y'\ran\ran,$ so
by Theorem \ref{U1}(3), $c$ is a primitive axis in~$\lan\lan
\calx'\ran\ran.$  Hence we may assume that $X'=\calx'$ and that
$X''=\Cl(X)_{\gd,\gl}.$  So $A$ is spanned by $X'\cup X''.$

We now use Lemma \ref{e}, to show that $c$ is a primitive axis in $A.$
First observe that
\[
\textstyle{\sum_{i=1}^k\ga_i x'_i=c=c^2=\sum_{i=1}^k\ga_i^2 x'_i+\sum_{i<j}\ga_i\ga_j(x'_i+x'_j).}
\]
Thus $\sum_{i=1}^k\ga_i=1.$
As $c$ is an primitive axis in $A',$ we have:
\begin{equation}\label{a'}
\text{Let $x'\in X'.$ Then either $x'\in A_0(c),$ or $x'-c\in A_{\gl,\gd}(c).$}
\end{equation}

Next let $x''\in X''.$
By (i),  $x'_jx''\in A_{\gl,\gd}(x'_i),$ for all $i,j.$
It follows that $cx''\in A_{\gl,\gd}(x'_i),$ for all $i,$
and since $\sum_{i=1}^k\ga_i=1,$ we get that $cx''\in A_{\gl,\gd}(c).$
Hence
\begin{equation}\label{a''}
\textstyle{\text{Let $x''\in X''.$ Then $cx''\in A_{\gl,\gd}(c),$ and $x''-\frac{1}{\gl}cx''\in A_0(c).$}}
\end{equation}
By equations \eqref{a'} and \eqref{a''} $c$ satisfies the hypothesis of Lemma \ref{e}(1).

We now check that $c$ satisfies the hypotheses of equation \eqref{eqe} of Lemma~\ref{e}(2).  Let $a,b$ be as in Lemma \ref{e}(2).
If we take $a,b\in X',$ then  the hypotheses of equation  \eqref{eqe} hold, because
$c$ is a primitive axis in $A'.$

Let now $a=x'\in X'$ and $b=x''\in X''.$
\smallskip

{\bf Case 1.}  $x'\in A_0(c).$

\noindent
In this case $x'_ix'=0,$ for all $i,$
and then, by (i), $x'(x'_ix'')=0,$ for all $i,$ so $x'(cx'')=0=(cx'')x'.$  Also,
by (i), $x'x'', x''x'\in A_0(x'_i),$ for all $i,$ so $x'x'', x''x'\in A_0(c).$
Hence
\[
\textstyle{x'(x''-\frac{1}{\gl}cx'')=x'x''\in A_0(c),\text{ and }(x''-\frac{1}{\gl}cx'')x'=x''x'\in A_0(c).}
\]
\smallskip

{\bf Case 2.} $x'-c\in A_{\gl,\gd}(c).$

\noindent
In this case, $x'x'_i\ne 0,$ for all $i.$
But then, by (i), $x'_ix''\in A_{\gl,\gd}(x'),$ so ${cx''\in A_{\gl,\gd}(x').}$
Thus, using also equation \eqref{a''},
\[
(x'-c)(cx'')=\gl cx''-\gl cx''=0=\gd cx''-\gd cx'' =cx''(x'-c).
\]
Also, by (i), $x'x''\in A_{\gl,\gd}(x'_i),$ for all $i,$
so $x'x''\in A_{\gl,\gd}(c).$ Hence
\[
\textstyle{(x'-c)(x''-\frac{1}{\gl}cx'')=(x'-c)x''\in A_{\gl,\gd}(c)\ni (x''-\frac{1}{\gl}cx'')(x'-c).}
\]

Finally assume $a=a''\in X''$ and $b=x''\in X''.$

\noindent
Then $ca'', cx''\in Z,$
so $(ca'')(cx'')=0.$  Also $(ca'')(x''-\frac{1}{\gl}ca'')=(ca'')x''.$  However,
by (i) (with $X''$ in place of $X'$), $x'_ia''\in A_0(x''),$ for all $i,$ if $a''x''=0,$ so also $ca''\in A_0(x''),$
and $x'_ia''\in A_{\gd,\gl}(x''),$ for all $i,$ if $a''x''\ne 0,$ so also $ca''\in A_{\gd,\gl}(x'').$
Now
\begin{equation}\label{dp}
\textstyle{(ca'')(x''-\frac{1}{\gl}cx'')=(ca'')x'',\text{ and }(x''-\frac{1}{\gl}cx'')(ca'')=x''(ca''),}
\end{equation}
so we see that if $ca''\in A_0(x''),$ then the two products in equation \eqref{dp} are zero,
while if $ca''\in A_{\gd,\gl}(x''),$ then these products are in $\ff (ca''),$ so
they both are in $A_{\gl,\gd}(c).$

Finally let
\[
\textstyle{v:=(a''-\frac{1}{\gl}ca'')(x''-\frac{1}{\gl}cx'')=a''x''-\frac{1}{\gl}a''(cx'')-\frac{1}{\gl}(ca'')x''.}
\]
If $a''x''=0,$ then we saw already that $v=0.$  Assume $a''x''\ne 0.$  Then, as above
$a''(cx'')=\gd cx'', (ca'')x''=\gl ca'',$ so since $ca'', cx''\in A_{\gl,\gd}(c),$ we get
\[
\textstyle{v=a''x''-\frac{\gd}{\gl}cx''-ca''=\gl a''+ \gd x''-\frac{\gd}{\gl}cx''- ca''.}
\]
It follows that $cv=\gl ca''+\gd cx''-\gd cx''-\gl ca''=0.$  This
completes the verification that $c$ satisfies the hypotheses of equation \eqref{eqe}.

We shall show in Theorem \ref{flex2} below that $A$ is flexible,
so we conclude that $c$ is a primitive axis.
\medskip

(vii)
Denote by $\chi'$ (resp.~$\chi''$) the set
of all primitive axes of type $(\gl,\gd)$ (resp.~$(\gd,\gl)$) of $A.$

Set $a'=a+z.$
Of course we may assume that $z\ne 0.$
We use Lemma \ref{e}.  Let $x'\in \chi'.$
If $ax'=0,$ then, by
(i), $zx'=0,$ so $a'x'=0,$ and $x'\in A_0(a').$
Suppose $ax'\ne 0,$ then $zx'=\gd x',$ by (i), so
\[
a'x'=ax'+\gd z=\gd a+\gl x'+\gd z=\gd a'+\gl x',
\]
That is $\lan\lan a', x'\ran\ran\cong U_2(\gl).$  Thus
$x'-a'\in A_{\gl,\gd}(a').$  We have shown:
\begin{equation}\label{x'}
\begin{aligned}
&\text{For $x'\in \chi'$ we have:}\\
&\text{either $x'\in A_0(a'),$ or $x'-a'\in A_{\gl,\gd}(a').$}
\end{aligned}
\end{equation}

Next let $x''\in \chi''.$
Then $z(a'x'')=0,$ as $a'x''\in Z,$ so by (i),
$a'x''=ax''+zx''\in A_{\gl,\gd}(a')\cap Z.$
Thus
\begin{equation}\label{x''}
\begin{aligned}
&\text{For $x''\in \chi''$ we have:}\\
&\text{$x''=\frac{1}{\gl}(\gl x''-a'x''+a'x''),$}\\
&\text{with $\gl x''-ax''\in A_0(a'),$ and $ax''\in A_{\gl,\gd}(a').$}
\end{aligned}
\end{equation}
It remains to check condition (2) of Lemma \ref{e}.

Let $x'_1, x'_2\in\chi'.$ Suppose first that $a'x'_1=0.$
If $a'x'_2=0,$ then either $x'_1x'_2=0,$ so of course $x'_1x'_2\in A_0(a'),$
or $x'_1x'_2=\gd x'_1+\gl x'_2\in A_0(a').$

Suppose $a'x'_2\ne 0,$  and $x'_1x'_2\ne 0.$  We claim that this is
impossible.  Indeed, by the proof of Theorem \ref{U1}(1), with
$x'_2$ taking the role of $c,$ $x'_1$ taking the role of $a$, and
$a'$ taking the role of $b,$ we see that $a'x'_1\ne 0.$

Thus we may assume that $a'x'_1\ne 0\ne a'x'_2.$  Again,
as above, we must have $x'_1x'_2\ne 0.$  But then
\begin{align*}
&(x'_1-a')(x'_2-a')=x'_1x'_2-x'_1a'-a'x'_2-a'\\
&=\gl x'_1+\gd x'_2-\gl x'_1-\gd a'-\gl a'-\gd x'_2-a'=0.
\end{align*}

Let $x'\in\chi'$ and $x''\in\chi''.$  Suppose first that $a'x'=0.$
Then $ax'=0,$ so by~(i), $x'(a'x'')=x'(ax''+zx'')=0,$ and $x'(\gl
x''-ax'')=\gl x'x''$.  But by (i) and (iii), $a(x'x'')=0$ and
$z(x'x'')=0,$ so $x'(\gl x''-ax'')\in A_0(a').$ Similarly  $(\gl
x''-ax'')x'\in A_0(a').$

Suppose next that $a'x'\ne 0.$  Then $ax'\ne 0,$ and by (iii), $z(a'x')=0,$ as $a'x''\in Z.$
Since $x'a\ne 0,$ and $az=\gl z,$ (i) implies that $x'(zx'')=\gl zx''.$ Also by (i),
$x'(ax'')=\gl ax'',$ hence $x'(a'x'')=\gl a'x''.$
Similarly $a'(a'x'')=a(a'x'')=\gl a'x'',$ thus $(x'-a')(a'x'')=0.$

Finally let $x''_1, x''_2\in\chi''.$  Since $a'x''_i\in Z$ for
$i=1,2,$ (iii) implies $$(a'x''_1)(a'x''_2)=0.$$  Note that by (i),
for $i=1$ or $2,$
 \[
(a'x''_i)x''_{3-i}=
\begin{cases}
0 &\text { if }x''_1x''_2=0,\\
\gl a'x''_i &\text{ if }x''_1x''_2\ne 0.
\end{cases}
\]
\[
x''_{3-i}(a'x''_i)=
\begin{cases}
0 &\text { if }x''_1x''_2=0,\\
\gd a'x''_i &\text{ if }x''_1x''_2\ne 0.
\end{cases}
\]
Hence
\[
a'x''_1(\gl x''_2-a'x''_2)=\gl(a'x''_1)x''_2=\gl\mu(a'x''_1)\in
A_{\gl,\gd}(a'),\quad\mu\in\{0,\gl\}.
\]
Similarly $(\gl x''_2-a'x''_2)a'x''_1\in A_{\gl,\gd}(a').$
Finally let
\begin{gather*}
v:=(\gl x''_1-a'x''_1)(\gl x''_2-a'x''_2)=\gl^2x''_1x''_2-\gl x''_1(a'x''_2)-\gl (a'x''_1)x''_2.
\end{gather*}
Now if $x''_1x''_2=0,$ then $v=0.$  Suppose $x''_1x''_2\ne 0.$
Then
\begin{gather*}
v=\gl^2(\gl x''_1+\gd x''_2)-\gl\gd a'x''_2-\gl^2 a'x''_1\\
=\gl^2(\gl x''_1-a'x''_1)+\gl\gd(\gl x''_2-a'x''_2)\in A_0(a').
\end{gather*}
We have shown that $a'$ satisfies all the hypotheses in Lemma \ref{e}, except
for $L_{a'}R_{a'}=R_{a'}L_{a'}.$
This will follow from Theorem \ref{flex2} below, which shows that $A$ is flexible.
\medskip

(viii)  Let $a\in X'$ and $b\in X''.$  It suffices to show that $ab\in A'_{\gl,\gd}(a),$
because by Lemma \ref{p1}, $A'_{\gl,\gd}(a)=A'_{\gl,\gd}(a'_j),$ for some $j.$
But by (vi) and (i), $a+ab\in X',$ so $ab=a+ab-a\in A'_{\gl,\gd}(a),$  since $a(a+ab)\ne 0.$
\medskip

(ix)
  Since $X'$ is connected, $A'_{\gl,\gd}(a)=\sum \ff_{x'\in X'}
(x'-a),$ by Theorem~\ref{U1}(2). By (i), $Z\subseteq
A_{\gl,\gd}(a),$ so by (iv), $A_{\gl,\gd}(a)$ is as claimed.  It is
easy to check now that $A_{\gl,\gd}(a)$ is an ideal. Since
$A'/A'_{\gl,\gd}(a)\cong\ff a,$ and $A/Z\cong
\widebar{A'}\times\widebar{A''},$ we see that $A/A_{\gl,\gd}(a)\cong
\ff\times \widebar{A''}.$ Hence the first paragraph of (viii) holds.

For the second paragraph, using (vii) we see that $Z$ is contained
in the first and second summand of equation \eqref{I}.
Hence by (ii),
\[
\textstyle{A/I\cong (A'/Z)\Big/\Big(\sum_{j\in J}A'_{\gl,\gd}(a'_j)/Z\Big)\times (A''/Z)\Big/\Big(\sum_{i\in I}A''_{\gl,\gd}(b'_i)/Z\Big).}
\]
Now (viii) follows from Lemma \ref{p1} and Theorem \ref{U1}.
\medskip

(x)
Let $u\in A$ satisfy $Au=0.$   Suppose $u\ne 0.$
Since $\widebar{A'}$ and
$\widebar{A''}$ contain no annihilating elements, $u\in Z,$ by (ii).
Write $u=\sum_{i=1}^m \ga_ix'_ix''_i,$ with $x'_j\in X',$ ${x''_j\in X''},$ $\ga_j\in\ff,$ and with $m$ minimal.
By (i), $0=x_1'u=\sum_{\{j\, :\, x'_1x'_j\ne 0\}}\gl \ga_j x'_jx''_j,$
so
\[
\textstyle{u=\sum_{i=1}^m \ga_i x'_ix''_i-\sum_{\{j\, :\, x'_1x'_j\ne 0\}} \ga_jx'_jx''_j},
\]
is a shorter sum, a contradiction.  Hence $u=0.$

The last assertion follows from Theorem \ref{miya1}.
 \end{proof}

\begin{thm}\label{dec2}
Every PAJ $A$ is a direct product of uniformly generated noncommutative PAJ's (described above) with a CPAJ.
\end{thm}
\begin{proof}
Let $X_i$ and $A_i$ be as in Theorem~\ref{miya1}.  By Remark \ref{ugen}, if
some axis in $X_i$ is of type $(\gl,\gd),$ with $\gl\ne\gd,$ then $X_i$ is uniform
of type $\{\gl,\gd\}.$  Since, by  Theorems \ref{U1} and \ref{almostcomm}(x)
for $X_i$ of type $\{\gl,\gd\}$ as above, $A_i$ contains no annihilating
elements, Theorem \ref{miya1}(iii) completes the proof.
 \end{proof}

Thus, after all this effort, the theory of PAJ's reduces to the
commutative case in \cite{HRS}.

\subsection{Flexibility}$ $
\medskip

We shall now show that every PAJ is flexible.
Since commutative algebras are always flexible, we
concentrate on noncommutative PAJ's.

2-generated PAJ's are flexible, by their classification in \cite{RoSe1},
see Remark \ref{ugen} and Example \ref{try1}.  But the question of
flexibility becomes nontrivial when we consider
three or more generators. Note that the linearization of flexibility
is
\begin{equation}\label{fle} (x_1 y)x_2 + ( x_2 y) x_1 = x_1 (y x_2) +  x_2 (y
x_1).\end{equation}

\begin{lemma}\label{flex1}
If $A$ is a PAJ generated by a set of primitive axes of type $(\gl,\gd),$ with
$\gl\ne\gd,$ then $A$ is flexible.
\end{lemma}
 \begin{proof}
By Theorem \ref{U1}, we may assume that $A:=\uu_X(\gl).$ Now
$\uu_X(\gl)$  is spanned by  the set $X$ of primitive axes of type
$(\gl,\gd),$ as defined in Definition~\ref{uuI}.
Hence it suffices to show that equation \eqref{fle} holds when
$|X|=3.$ Then it is enough to assume that $X=\{x_1, x_2,y\}.$

Thus $x_i y = \gd x_i + \gl y$ and $y x_i = \gl x_i + \gd y$ for
primitive axes $x_i,y$ all of type~$(\gl,\gd)$, $i = 1,2.$
Hence
\begin{equation}\begin{aligned}(x_1 y)x_2
+ ( x_2 y) x_1  & =  (\gd x_1 + \gl y)x_2 +  (\gd x_2 + \gl y) x_1
\\ &
= \gd ( x_1  x_2 + x_2 x_1) +2 \gl \gd y + \gl^2 (x_2 +  x_1)
.\end{aligned}\end{equation}

\begin{equation}\begin{aligned}
 x_1 (y x_2) + x_2 (y x_1) &= x_1(\gd y + \gl x_2) + x_2(\gd y + \gl x_1) \\& =
\gl  (x_1 x_2+x_2 x_1)+ 2 \gl \gd y + \gd^2 (x_1
+x_2).\end{aligned}
\end{equation}

The difference is
\begin{equation*}
\begin{aligned}
&(\gl^2- \gd^2)(x_1+x_2)+ (\gd - \gl)  (x_1 x_2+x_2 x_1)\\
&=(\gd - \gl)(x_1x_2+x_2x_1 - (x_1+ x_2)),
\end{aligned}
\end{equation*}
which is $0$ since
\[
x_1x_2 + x_2 x_1 = (\gl + \gd)(x_1 + x_2) = x_1+x_2.\qedhere
\]
 \end{proof}

\begin{thm}\label{flex2}
Any PAJ $A$ is flexible.
\end{thm}
\begin{proof}
By Theorem \ref{KMt}(i), it is enough to verify \eqref{fle} for
three primitive axes. Since any commutative algebra is flexible,
in view of Corollary~\ref{dec2}, we may assume these axes all are of
some noncommutative type $(\gl,\gd)$ or $(\gd,\gl),$  for  some
$\gl$, with $\gd = 1-\gl,$ and
that these three axes are in the same connected component of the
axial graph.

If all three axes have the same type we are done by
Lemma~\ref{flex1}. Hence we may assume that two  have type
$(\gl,\gd)$ and one has type $(\gd,\gl).$

We will use repeatedly  the fact that:
\[
\text{In the algebra $A_{\operatorname{exc},3}(\{a,b\},\gl)$ of Example \ref{try1}, $\gd ab=\gl ba,$}
\]
as well as the structure of the algebra in Remark \ref{ugen}.
\smallskip

\noindent
{\bf Case I}.  $ab\ne 0,$ for all $a,b\in \{x_1,x_2,y\}.$
\smallskip

\noindent
{\bf Subcase I1}.  $x_1, x_2$ are of type $(\gl,\gd),$ and $y$ is of type $(\gd,\gl).$
\smallskip

\noindent
By Theorem \ref{almostcomm}(i),  $x_iy, yx_i\in A_{\gl,\gd}(x_{3-i}),$  for $i\in \{1,2\}.$ Hence
\begin{align*}
 (x_1 y) x_2 + (x_2 y) x_1&=\gd x_1y+\gd x_2y=\gl yx_2 + \gl yx_1=\\
& x_1 (y x_2) + x_2 (y x_1).
\end{align*}
\smallskip

\noindent
{\bf Subcase I2}.  $x_1, y$ are of type $(\gl,\gd),$ and $x_2$ is of type $(\gd,\gl).$
\smallskip

\noindent
Again we have
\begin{align*}
&(x_1 y) x_2 + (x_2 y) x_1=(\gd x_1+\gl y)x_2+\gd (x_2y)\\
&=\gd (x_1x_2)+\gl (yx_2)+\gd (x_2y)=\gl (yx_2)+\gd(x_2y)+\gl (x_2x_1)\\
&=x_1 (y x_2) + x_2 (\gd y+\gl x_1)=x_1 (y x_2)+x_2(yx_1).
\end{align*}
\medskip

\noindent
{\bf Case II}.  $ab= 0,$ for some $a,b\in \{x_1,x_2,y\}.$
\smallskip

\noindent
{\bf Subcase II1}.  $ac=0,$ for all $c\in\{x_1,x_2,y\}\sminus\{a\}.$
\smallskip

\noindent
Of course we may assume that $a=x_1,$ and that $yx_2\ne 0.$ We must show that
\begin{equation}\label{=}
(x_2y)x_1=x_1(yx_2).
\end{equation}
If $x_2$ and $y$ have the same type, then $x_2y$ and $yx_2$ are
linear combinations of~$x_2$ and $y,$ so both sides in equation
\eqref{=} are equal to~$0.$  If $x_2$ and $y$ are of different
types, then, by Theorem \ref{almostcomm}(i), both sides of equation
\ref{=} equal~$0.$
\smallskip

\noindent
{\bf Subcase II2}.  $x_1, x_2$ are of type $(\gl,\gd),$ and $y$ is of type $(\gd,\gl).$
\smallskip

\noindent
If $x_1x_2=0,$ then, by Theorem \ref{almostcomm}(i), $(x_iy)x_j=x_i(yx_j)=0,$ for $\{i,j\}=\{1,2\},$
so  equation \eqref{fle} holds.

Suppose $x_1y=0.$  Then
\begin{align*}
(x_1 y) x_2 + (x_2 y) x_1&=(x_2 y) x_1=\gd (x_2y)=\gl (yx_2)\\
&=x_1 (y x_2)=x_1 (y x_2) + x_2 (y x_1).
\end{align*}
\smallskip

\noindent
{\bf Subcase II3}.  $x_1, y$ are of type $(\gl,\gd),$ and $x_2$ is of type $(\gd,\gl).$
\smallskip

\noindent
If $x_1y=0,$ then
\begin{align*}
&(x_1 y) x_2 + (x_2 y) x_1=(x_2 y) x_1=\gd (x_2y)\\
&=\gl (yx_2)=x_1(yx_2)=x_1(yx_2)+x_2(yx_1).
\end{align*}
By symmetry, equation \eqref{fle} holds if $x_2y=0.$

Finally assume $x_1x_2=0,$ then
\begin{align*}
&(x_1 y) x_2 + (x_2 y) x_1=(\gd x_1+\gl y)x_2+\gd (x_2y)=\gl (yx_2)+\gd (x_2y)\\
&=\gl (yx_2)+x_2(\gd y+\gl x_1)=x_1(yx_2)+x_2(yx_1).
\end{align*}
\end{proof}

\begin{cor}\label{KMt01}
For any PAJ $A,$ and any idempotent $a\in A,$ we have
$L_a^2 - L_a = R_a^2 - R_a$.
\end{cor}
\begin{proof}  Since $A$ is flexible, this follows from \cite[Lemma~2.3]{RoSe2}.
\end{proof}

\section{Properties of primitive axes}\label{lowdim}$ $

In this section we delve deeper into properties of primitive axes. This
depends on our determining the structure of 2-generated PAJ's,
having classified the noncommutative ones.

The following concept is useful (and will be needed) for arbitrary
PAJ's. Following \cite[P.~95]{HRS}, for any primitive axis $a$ we define the function
\[
\gvp_a\colon A\to\ff,
\]
where $\gvp_a(y)a$ is the projection of $y$ onto the 1-eigenspace
$\ff a$ of $a$.

\begin{prop}\label{sym}
$\gvp_a(b)=\gvp_b(a)$ for any primitive axes $a,b\in A$ under any of
the following conditions:
 \begin{enumerate}\eroman
\item
$ab = 0$, in which case $\gvp_a(b)=\gvp_b(a) = 0;$

\item
$a$ or $b$ has type $(\gl,\gd)$ with $\gl \ne \gd;$

\item
$\langle\langle a,b\rangle\rangle$ is commutative, and $a,b$ are of the same type.
\end{enumerate}
\end{prop}

\begin{proof} (i) Since $ab=0,$ also $ba=0,$ so the projections are $0.$

(ii) We may assume that  $a$ is of  type $(\gl,\gd).$ By Remark
\ref{ugen}, if $b$~has type~$(\gl,\gd)$, then $\gvp_a(b) = \gvp_b(a)
= 1$ since $b = a +(b-a),\ a=b+(a-b),$ and $a-b\in
A_{\gl,\gd}(a)\cap A_{\gl,\gd}(b).$ Otherwise, by Example
\ref{try1}, $\gvp_a(b) = \gvp_b(a) = 0$ since $b = (b-y) + y$ and
$a=(a-y)+y.$

(iii) This appears in \cite{HRS}.
 \end{proof}

Equality  fails in $B(\gl, \varphi)$ (see Example \ref{non}(i)), when $\gl \ne \half.$
\begin{lemma}\label{uni25}
Let $A=B(\gl,\gvp),$ and assume $a, b\in A$ are primitive axes, with
$a$ of type~$\gl$ and  $b$ of type $1-\gl.$ If
$\gvp_a(b)=\gvp_b(a),$ then $\gl = \half.$
 \end{lemma}
 \begin{proof}
Assume that $\gl\ne\half.$  By \cite[Prop.~C]{RoSe1}, $\dim A=3.$
By \cite[Proposition 2.10(4)]{RoSe1}, taking $\gl' = 1-\gl$, we have
\[
(\gvp_a(b) -1)(1-\gl) = (\gvp_b(a)-1)\gl = \gc.
\]
Suppose that $\gvp_a(b)=\gvp_b(a).$
If $\gvp_a(b)=1,$ then $\gc=0,$ so by \cite[Proposition 2.12(iii)]{RoSe1},
$\gl=\gl'.$  Otherwise $\gl=\half,$ a contradiction.
\end{proof}

Ironically, the structure of the
CPAJ's is more complicated than the noncommutative PAJ's, which we
have already classified in Theorem~\ref{almostcomm}. The first step
is to determine all of the idempotents of a 2-generated CPAJ.

\subsection{The idempotents of CPAJ's of dimension $2$}$ $
\medskip

We have already
considered the easy case of noncommutative PAJ's  of dimension $2$.

\begin{lemma}\label{dim2}
Suppose that $X=\{a,b\}$ with $ab\ne 0,$
that $A=\lan\lan X\ran\ran$ is a CPAJ of dimension $2$
and that $a$ is of type $\gl.$   Then $b$ is also of type $\gl,$ and
\begin{enumerate}
\item
$\gl\in\{-1,\half\}$ and $ab=\gl a+\gl b.$

\item
If $\gl=\half,$ the idempotents in $A$ are $c_{\mu}:=\mu a+(1-\mu)b,\ \mu\in\ff,$
and they are all primitive axes of type $\half,$ and $\lan\lan a, c_{\mu}\ran\ran=A,$ for $c_\mu\ne a.$
Letting $x=a-b,$ we have
\[\tag{$*$}
a^{(\gt_b\gt_a)^k}=a+2kx,\qquad a^{(\gt_b\gt_a)^k\gt_b}=b-(2k+1)x,
\]
and $a^{\gt_{c_{\mu}}} = 2c_{\mu}-a =(2\mu -1)a + 2(1-\mu)b,$
in particular $a^{\gt_{1/2}}=b.$

\item
If $\gl=-1\ne\half,$ then the idempotents in $A$ are $a,b,-a-b,$ they are primitive axes of type $-1,$ and
$-a-b = b^{\gt_a} = a^{\gt_b}.$
\end{enumerate}
\end{lemma}
\begin{proof}
By \cite[Proposition C(1)]{RoSe1}, $b$ is also of type $\gl.$
\medskip

(1)    Appears in \cite{HRS}.
\medskip

(2)  By Lemma \ref{p1} the primitive axes in $A$ are as given. ($*$)
holds by \cite[Remark~2.14]{RoSe1}.  Let $c:=c_{\mu}.$ By Lemma
\ref{p1}, $a=c+a-c,$ with $a-c\in A_{1/2}(c),$ so $a^{\gt_c} = 2c-a
= (2\mu -1)a + 2(1-\mu)b.$
\medskip

(3)  See \cite[Lemma 3.1.8]{HSS}.
\end{proof}

\begin{example}\label{Uprime}
 \cite{HRS} gives names to $\ff \times
\ff$ and $\uu_2(\half)$. In its classification of CPAJ's of
dimension 2, \cite{HRS} defines one more 2-dimensional algebra
$3C(-1)^\times$, as the algebra spanned by primitive axes $a$ and $b$ with $ab
=-a-b.$  So in analogy to Definition~\ref{uuI},  we define an
algebra $\uu_E' $  to be the algebra having as a basis the idempotents
$E:=\{ e_i: i \in I \}$, with multiplication given by
\[
e_ie_j =   -e_i  - e_j , \qquad i\ne j.
\]
Here the $(-1)$-eigenspace of any idempotent $a\in E$
is $\sum_{a\ne e\in E} \ff(a+2e)$ since
\[
a(a+2e) = a - 2a - 2e =-(a+2e).
\]
It is easy to check that $a$ is a primitive axis in $U'_E.$
Thus $3C(-1)^{\times} = \uu_{\{a,b\}}'$.
\end{example}

\subsection{The idempotents of 2-generated
CPAJ's of dimension $3$}\label{sub 2-gen}$ $
\medskip

We can also obtain rather precise information about the primitive
axes of 2-generated CPAJ's of dimension 3. Suppose throughout this subsection that
$A = \langle\langle a,b\rangle\rangle $ is a CPAJ with $\dim A=3,$
where $a,b$ are primitive axes of respective types $\gl, \gl'$. As
in \cite{RoSe1}, inspired by \cite{HRS} we let
\[
\gs =ab-\gl'a-\gl b
\]
and
\begin{equation}\label{eqg}
\gc:=\ga_b(1-\gl)-\gl'=\gb_a(1-\gl')-\gl\in \ff,
\end{equation}
where $\ga_b=\gvp_a(b),$ and $\gb_a=\gvp_b(a),$ cf.~\cite[Proposition 2.10(iv)]{RoSe1}.
By  \cite[Proposition~2.12]{RoSe1},  $A : = \ff a + \ff b + \ff \gs,$
and $\gs\ne 0,$ since $\dim A=3.$

 Recall from \cite[Proposition 2.12(v)]{RoSe1} that $\gl = \gl'$
or $\gl+\gl' = 1.$   If $\gl = \half$, then all primitive axes of $A$ are
of type $\half.$

\begin{lemma}\label{gc=0}
Assume $\gc=0.$  Then
\begin{enumerate}
\item
$\gl=\gl'\in\{\half, -1\}.$

\item
If $\gl=\half,$ then the idempotents in $A$ are
\[
c_{\mu}:=\mu a+(1-\mu)b+2\mu(1-\mu)\gs,
\]
which are all primitive axes of type $\half,$ and $\lan\lan a, c_{\mu}\ran\ran=A,$ for all $c_{\mu}\ne a.$
We have  $a^{\gt_c}=b,$ where $c=\half a+\half b+\half \gs.$

\item
If $\gl=-1\ne\half,$ then the idempotents in $A$ are $a, b, -a-b+2\gs.$
These are all primitive axes of type $-1,$ and $c=a^{\gt_b}=b^{\gt_a},$ for $c=-a-b+2\gs.$
\end{enumerate}
\end{lemma}
\begin{proof}
(1)  This is \cite[Prop.~2.12(iii)]{RoSe1}.
\medskip

(2)
Since $\gs$ is annihilating,
passing modulo~$\ff\gs$, we  get from Lemma \ref{dim2}(2), that any idempotent must
have image $\mu \bar a + (1-\mu)\bar b$, so the idempotent is of the
form $\mu a + (1-\mu) b+\nu \gs,$  for $\mu,\nu\in\ff.$
We have
\begin{equation}
\begin{aligned}
&\mu  a + (1-\mu) b+\nu\gs   =(\mu a + (1-\mu) b)^2\\
&=\textstyle{ \mu^2 a +(1-\mu)^2 b + 2\mu(1-\mu)(\gs+\half a +\half b)}\\
& =(\mu^2+ \mu(1-\mu))a + ((1-\mu)^2+  \mu(1-\mu))b+2\mu(1-\mu)\gs\\
&=\mu a+(1-\mu)b+2\mu(1-\mu)\gs.
\end{aligned}
\end{equation}
 Hence $\nu = 2\mu(1-\mu).$  It is easy to check that for
\[
c_{\mu}:=\mu a+(1-\mu)b+2\mu(1-\mu)\gs,
\]
we have $A_{\half}(c_{\mu})=\ff\big(a-b+2(1-2\mu)\gs\big).$
Then $A_1(c_{\mu})=\ff c_{\mu},\ A_0(c_{\mu})=\ff\gs.$
Also,
\begin{gather*}
\textstyle{(a-b+2(1-2\mu)\gs)^2=(a-b)^2=a-2ab+b}\\
\textstyle{=a-2(\gs+\half a+\half b)+b=-2\gs\in A_0(c_{\mu}).}
\end{gather*}
Hence, $c_{\mu}$ satisfies the fusion rules and we see that every
idempotent in $A$ is a primitive axis in $A.$  Suppose $c_{\mu}\ne a,$
and set $B:=\lan\lan a, c_{\mu}\ran\ran.$
Then since $ac_{\mu}\in B,$ we see that $ab\in B,$ and then it is easy to check that
$a,b,\gs\in B,$ so $B=A.$ We saw that the type of $c_{\mu}$ is
$\half.$

 Let $c=\half a+\half b+\half \gs.$  Then $A_{\half}(c)=a-b,$ and $a=c-\half\gs+\half(a-b).$
Hence $a^{\gt_c}=c-\half\gs-\half(a-b)=b.$
\medskip

\noindent
(3)  The calculations are easy and we omit them.
\end{proof}

\begin{remark}\label{annaxis}
$A$ has a nonzero annihilating element iff $\gc =0.$ Indeed if
$y\ne 0$ is any annihilating element, then $0 = y \gs = \gc y,$ so
$\gc = 0.$
\end{remark}

\begin{prop}\label{c+}
\begin{enumerate}
\item
$b_{\gl}=\frac{(\gl'-\ga_b)}{\gl}a+b+\frac{\gs}{\gl}.$

\item
Suppose $\gc=0.$  Then
\begin{itemize}
\item[(i)]
$\gl=\gl'\in\{\half,-1\}.$

\item[(ii)]
$\ga_b=\gb_a=\gl+\half.$

\item[(iii)]
$b_{\gl}^2=(\frac{1}{4\gl^2}-1)a-\frac{1}{\gl}\gs.$
\end{itemize}

\item
Suppose $\gc\ne 0.$  Then
\begin{itemize}
\item[(i)]
$\one:=\frac{\gs}{\gc}$ is the identity element of~$A,$ and if
$\gl'\ne\gl,$ then ${\gl'=1-\gl.}$

\item[(ii)]
We have
\begin{equation}\label{bl1}
\begin{aligned}
\textstyle{b_{\gl}^2=\Big(\frac{(\gl'-\ga_b)^2}{\gl^2}+\frac{2\gl'(\gl'-\ga_b)}{\gl}+\frac{2(\gl'-\ga_b)\gc}{\gl^2}\Big)a}\\
\textstyle{+\underbrace{\textstyle{(1+2\gl'-2\ga_b+\frac{2\gc}{\gl})b}}_{=0}+\Big(\frac{\gc^2}{\gl^2}+\frac{2(\gl'-\ga_b)\gc}{\gl}\Big)\one.}
\end{aligned}
\end{equation}

\item[(iii)]
If $\gl=\gl',$ then $\ga_b=\gb_a.$

\item[(iv)]
If $\gl=\half,$ then $\ga_b=\gb_a$ can be any number.

\item[(v)]
If $\gl=\gl'\ne\half,$ then  $\gl\ne -1,$ $\ga_b=\gb_a=\half\gl$ and $\gc=-\half\gl(\gl+1).$

\item[(vi)]
If $\gl'=1-\gl\ne \half,$ then
\[
\textstyle{\ga_b=1-\half\gl,\quad\gb_a=1-\half\gl',\quad\gc=\half\gl(\gl-1).}
\]
\end{itemize}
\end{enumerate}
\end{prop}
\begin{proof}
(1)
We have
\[
\ga_ba+\gl b_{\gl}=ab=\gs+\gl'a+\gl b,
\]
so $b_{\gl}$ is as claimed.
\medskip

\noindent
(2)
(i)
This is \cite[Proposition 2.12(iii)]{RoSe1}.
\medskip

\noindent
(ii)
By (1), and (i), and since $\gs v=0,$  for all $v\in A,$

 \begin{equation}\begin{aligned}\notag
 b_{\gl}^2& =\frac{(\gl-\ga_b)^2}{\gl^2}a+b+\frac{2(\gl-\ga_b)}{\gl}ab\\
\notag &=
\frac{(\gl-\ga_b)^2}{\gl^2}a+b+\frac{2(\gl-\ga_b)}{\gl}(\gs+\gl
a+\gl b) \\ \label{bl0} &=
\Big(\frac{(\gl-\ga_b)^2}{\gl^2}+\frac{2\gl(\gl-\ga_b)}{\gl}
\Big)a+(1+2(\gl-\ga_b))b +\frac{2(\gl-\ga_b)}{\gl}\gs.
\end{aligned}\end{equation}
By the fusion rules the coefficient of $b$ in equation \eqref{bl0}
must be $0,$ so $2(\gl-\ga_b)=-1.$ So $\ga_b-\gl=\half$ and
symmetrically, $\gb_a-\gl=\half;$ hence $\ga_b=\gb_a.$  Substitution
$\ga_b=\gl+\half$ in \eqref{bl0} we get
\[
\textstyle{b_{\gl}^2=(\frac{1}{4\gl^2}-1)a-\frac{1}{\gl}\gs.}
\]
\medskip

\noindent
(3)
(i)  This is \cite[Proposition 2.12(v)\&(vi)]{RoSe1}.
\medskip

\noindent
(ii)  Since $\gs=\gc\one,$ we get from (1) that
\begin{gather*}
\textstyle{b_{\gl}^2=\frac{(\gl'-\ga_b)^2}{\gl^2}a+b+\frac{\gc^2}{\gl^2}\one+\frac{2(\gl'-\ga_b)}{\gl}ab+\frac{2(\gl'-\ga_b)\gc}{\gl^2}a
+\frac{2\gc}{\gl}b},
\end{gather*}
so since $ab=\gc\one+\gl' a+\gl b,$ equation \eqref{bl1} holds.
\medskip

\noindent
(iii) This follows from equation \eqref{eqg}.
\medskip

\noindent
(iv) Since $b_{\gl}^2\in\ff a+\ff\one,$ the coefficient of $b$ in equation \eqref{bl1} is $0.$ So
\begin{equation}\label{bl2}
\begin{aligned}
&\gl+2\gl\gl'-2\ga_b\gl+2\gc=0,\text{ iff}\\
&\gl+2\gl\gl'-2\ga_b\gl+2\big(\ga_b(1-\gl)-\gl'\big)=0,\text{ iff}\\
&\gl+2\gl\gl'-2\gl'=4\ga_b\gl-2\ga_b.
\end{aligned}
\end{equation}
If $\gl=\half,$  then, by \cite[Proposition 2.12(iv)]{RoSe1}, $\gl'=\half,$ so by
part 3(iii), $\ga_b=\gb_a.$  We then see that equation \eqref{bl2} holds automatically, so
there are no other restrictions on $\ga_b.$
\medskip

\noindent
(v)
By equation \eqref{bl2} we see that
\begin{equation}\label{gab}
\textstyle{\ga_b=\frac{\gl+2\gl\gl'-2\gl'}{4\gl-2},}
\end{equation}
so if $\gl=\gl',$ we get $\ga_b=\gb_a=\half\gl.$  The formula for
$\gc$ comes from equation~\eqref{eqg}.  But now if $\gl=-1,$ then $\gc=0,$
a contradiction.
\medskip

\noindent (vi) This follows from equation \eqref{gab}, from
symmetry, and from equation~\eqref{eqg}.
\end{proof}

\subsection{The case where $\gc\ne 0$}$ $
\medskip

From this point to the end of subsection \ref{sub 2-gen}, we assume
that $\gc\ne 0.$ Hence $\one:=\frac{\gs}{\gc}$ is an identity element
of $A.$ Note that $\one-a$ is an axis in $A$ of type~$1-\gl$, with
$A_0(\one-a)=\ff a$ and $A_{1-\gl}(\one-a)=A_{\gl}(a).$ We let
\[
b=\ga_b a + \ga'_b (\one - a)+b_{\gl},\quad \ga'_b\in\ff
\]
be the decomposition of $b$ with respect to $a.$  We define
\begin{equation}\label{eq011}
w_b = \ga_b a + \ga'_b (\one - a)  =(\ga_b - \ga'_b) a +  \ga'_b\one.
\end{equation}
Then
\begin{equation}\label{eq012}
w_b^2 =  \ga_b^2 a +{\ga'_b}^2 (\one - a) =
( \ga_b^2 -{\ga'_b}^2 ) a+ {\ga'_b}^2 \one,
\end{equation}

 \begin{equation}\label{eq013}
b_\gl = b - w_b = b + (\ga'_b - \ga_b) a -\ga'_b\one.
\end{equation}

\begin{lemma}\label{2ndaxis}
Write $b =w_b +b_\gl,$ as above, with $b_\gl \in A_\gl(a).$
\begin{enumerate}\eroman
\item
If $\gl'=\gl\ne\half,$  then $\ga_b=\half\gl$ and $\gl\ne -1.$ If $\gl'=1-\gl\ne\half,$
then $\ga_b=1-\half\gl.$

\item
For ${w_e:=\ga_e a+\gb_e(\one-a),}$ an element $e:=w_e+\gr_eb_{\gl}$
is an idempotent in $A,$  iff
\begin{align}\label{5.8.1}
\gb_e &= \textstyle{\frac{1-2\gl \ga_e}{2(1-\gl)},\text{ and}}\\\label{5.8.2}
\gr_e^2b_\gl^2 &= (\gb_e -\gb_e^2) \one +  (\ga_e -\ga_e^2-(\gb_e -{\gb_e}^2 ) ) a.
\end{align}

\item
Suppose $\gl\ne\half,$ then  $\ga'_b=\frac{\gl+1}{2}$
if $\gl'=\gl$ and $\ga'_b=\frac{1-\gl}{2},$ if $\gl'=1-\gl.$

\item
If $\gl\ne\half,$ then in both cases where $\gl=\gl'$ and where $\gl\ne\gl',$
\begin{gather*}
\textstyle{\mu_b:=\ga'_b-{\ga'_b}^2=\frac{1-\gl^2}{4},\quad  \nu_b:=\ga_b-\ga_b^2=\frac{\gl(2-\gl)}{4},\text{ and}}\\
\textstyle{b_{\gl}^2=\frac{(1-\gl^2)}{4}\one+\frac{(2\gl-1)}{4}a.}
\end{gather*}
\end{enumerate}
\end{lemma}
\begin{proof}
(i)  This is Proposition \ref{c+}(3) parts (v) and (vi).
\medskip

(ii) $e$ is an idempotent iff
$w_e + \gr_eb_\gl = e = e^2 = w_e^2 + 2\gr_e w_e b_\gl  + \gr_e^2b_\gl ^2.$ The fusion rules show that
\[
\gr_eb_\gl =2\gr_ew_e b_\gl =  2\gr_e\big(\gl\ga_e b_\gl +(1-\gl)\gb_e b_\gl\big)=
\gr_e\big(2\gl\ga_e+2\gb_e(1-\gl)) b_\gl,
\]
 implying $1 =2\gl\ga_e +2\gb_e(1-\gl),$ so equation \eqref{5.8.1} holds.
Also
\begin{gather*}
\gr_e^2b_\gl^2 =w_e-w_e^2=\ga_e a+\gb_e(\one-a)  -( \ga_e^2 -{\gb_e}^2) a -{\gb_e}^2 \one \\
(\gb_e -{\gb_e}^2) \one +  (\ga_e - \ga_e^2-( \gb_e-{\gb_e}^2 ) ) a,
\end{gather*}
so (ii) holds.
\medskip

(iii)  This follows from equation \eqref{5.8.1} and (i).
\medskip

(iv)
The calculations for $\mu_b$ and $\nu_b$ come from (i) and (iii).
The calculation for $b_{\gl}^2$ come from applying (i) and (iii) to equation \eqref{5.8.2},
since $\gr_b=1.$
\end{proof}

\subsubsection{\underline{The case where $\gc\ne 0$ and  $\gl \ne \half$}}$ $
\medskip

\begin{thm}\label{nothalf}
Suppose that $\gc\ne 0$ and
$\gl \ne \half.$ Then the idempotents in $A$ distict from $0,\one$ are
$a,\, \one-a,\, b,\, \one-b,\, b^{\gt_a}$ and $\one-b^{\gt_a}.$
All these idempotents are primitive axes, and if $\gl'=\gl,$
then $a^{\gt_b}=b^{\gt_a},$ while if $\gl'=1-\gl,$ then $a^{\gt_b}=\one-b^{\gt_a}.$
\end{thm}
\begin{proof}
We use the notation of Lemma \ref{2ndaxis}.
Let $e=w_e+\gr_eb_{\gl}$ be an idempotent in $A$ as in Lemma \ref{2ndaxis}(ii),
and assume that $e\notin \{0,\one, a,\one-a\}.$
Notice that by equation \eqref{5.8.2} (first for $b$ and then for $e$),
\begin{equation}\label{xi}
\gr_e^2\mu_b=\gb_e-\gb_e^2\text{ and } \gr_e^2\nu_b=\ga_e-\ga_e^2.
\end{equation}
Also, $\gr_e\ne0,$ as $e\notin\{0,\one, a,\one-a\}.$  We show that
there are at most two possibilities for~$\ga_e.$  Then by equation
\eqref{5.8.1}, $\gb_e$ is determined and $\gr_e$ is determined up to
a sign, by equation \eqref{5.8.2}.  But there are two  distinct
possibilities for $\ga_e,$ these are $\ga_b$ and  $1-\ga_b.$ Indeed
\[
\one-b=(1-\ga_b)a+(1-\ga'_b)(\one-a)-b_{\gl},
\]
and $\ga_b\ne\half,$ by Lemma \ref{2ndaxis}(i).  This will show that
${e\in\{ b, b^{\gt_a}, \one-b, \one-b^{\gt_a}\}.}$

Suppose first that $\mu_b=0.$  Then $\gl=-1,$ and $\gb_e=0$ or $1,$
and then by~equation \eqref{5.8.1}, $\ga_e$ is determined.

Suppose next that $\gl\ne -1.$  In particular, $\mu_b\ne 0.$ Set
$\xi_e:=\frac{\nu_b}{\mu_b}=\frac{\gl(\gl-2)}{\gl^2-1}.$ By \eqref{xi},
\begin{gather*}
\textstyle{\ga_e^2-\ga_e=\xi_e(\gb_e^2-\gb_e)=\xi_e\Big(\frac{(1-2\gl \ga_e)^2}{4(1-\gl)^2}-\frac{(1-2\gl \ga_e)}{4(1-\gl)}\Big)}\\
=\textstyle{\xi_e\Big(\frac{\gl^2}{(\gl-1)^2}(\ga_e^2-\ga_e)+\frac{1}{4(\gl-1)^2}-\frac{(1-2\gl \ga_e)}{4(1-\gl)}\Big).}
\end{gather*}
Thus $\ga_e$ solves a quadratic equation whose coefficient of
$\ga_e^2$ is ${\xi_e \frac{\gl^2}{(\gl-1)^2}-1.}$ So if that
coefficient is zero, then
\begin{equation}\label{xii}
\textstyle{\frac{\gl(\gl-2)}{\gl^2-1}=\xi_e=\frac{(\gl-1)^2}{\gl^2}.}
\end{equation}
so $\gl^4-2\gl^3=(\gl^2-2\gl+1)(\gl^2-1)=\gl^4-2\gl^3+2\gl-1.$
Since $\gl\ne \half,$ the equality in equation \eqref{xii} does not hold,
so $\ga_e$ has at most two solutions.

Now if $\gl=\gl',$ then $a^{\gt_b}\in\{a, b, b^{\gt_a}\}.$  Further,
$a^{\gt_b}\ne b,$ as $a\ne b.$ Also $a^{\gt_b}\ne a,$ since by
\cite[Lemma 2.8(2)]{RoSe1}, $a_{\gl'}\ne 0$ (here $a_{\gl'}$ is the
projection of $a$ on~$A_{\gl'}(b)).$  Hence $a^{\gt_b}=b^{\gt_a}.$
A similar argument works when $\gl+\gl'=1.$
\end{proof}

\subsubsection{\underline{The case where $\gc\ne 0,$ and  $\gl =\half$}}$ $
\medskip

We are left with $\gl = \half,$ which in the set-up of \cite{HRS} is
the algebra $B(\half,\ga_b)$ and is rather different.
By \cite[Proposition 2.12(iv)]{RoSe1}, $\gl'=\half.$
Also, $\gc =\half (\ga_b -1),$ which is  nonzero (since $\gc = 0$ is Case II),
implying $\ga_b \ne 1.$ Also $\ga_b=\gvp _a (b) = \gvp _b (a)=\gb_a,$ by
Proposition~\ref{sym}.

We exploit the primitive axis $a' := \one -a$.
\begin{thm}\label{half}
Let $a':=\one-a,$ and set $\nu_b =\ga_b(1-\ga_b).$
\begin{enumerate}\eroman
\item
An element $e:=\ga_e a+\gb _e(\one-a)+\gr_e b_{\gl}$ is a non-trivial
idempotent in $A$ iff $\gb_e=1-\ga_e,$ and
\begin{equation}\label{rho}
\gr_e^2 b_{\gl}^2=(\ga_e-\ga_e^2)\one.
\end{equation}
These idempotents exist iff either
\begin{itemize}
\item[(1)]
$\ga_b=0,$ and then $\ga_e\in\{0,1\}$ and $\gr_e$ is arbitrary, or

\item[(2)]
$\ga_b\ne 0$ and $\frac{\ga_e(1-\ga_e)}{\nu_b}$ is a square in $\ff,$
in which case $\gr_e=\pm\sqrt{\frac{\ga_e(1-\ga_e)}{\nu_b}}.$
\end{itemize}

\item
$b=\ga_b a+(1-\ga_b)(\one-a)+b_{\gl},$ with $b_{\gl}^2=\nu_b\one.$
Any idempotent $e\in A$ is a primitive axis with $A_1(e)=\ff e,$ $A_0(e)=\ff (\one-e),$
and $A_{\half}(e)=\ff x,$ with
\[
x=4\ga_e(1-\ga_e)a-2\ga_e(1-\ga_e)\one+(1-2\ga_e)\gr_eb_{\gl},
\]
if $e\notin\{a,\one-a\},$ while $A_{\half}(e)=\ff b_{\gl},$ if $e\in\{a,\one-a\}.$
\end{enumerate}
\end{thm}
\begin{proof}
By Lemma \ref{2ndaxis}(ii) (equation \eqref{5.8.1}), $\gb_e=1-\ga_e.$  But then $\ga_e-\ga_e^2=\gb_e-\gb_e^2,$ so
by equation \eqref{5.8.2}, equation \eqref{rho} holds.

Note that since $b$ is an idempotent,
$b=\ga_ba+(1-\ga_b)(\one-a)+b_{\gl},$ and $b_{\gl}^2=\nu_b\one,$ by  equation \eqref{rho}, because $\gr_b=1.$
Hence, if $\ga_b=0,$ then $\nu_b=0,$ so by \eqref{rho}, $\ga_e\in\{0,1\},$ and there are no restrictions on $\gr_e,$
while if $\ga_b\ne 0,$ then $\nu_b\ne 0,$ and the restriction on $\gr_e$ comes from  \eqref{rho}.

Next we show that $e$ is a primitive axis in $A.$  We claim that
\begin{equation}\label{ae}
\textstyle{ae =\half(\ga_e-1)\one+\half a+\half e.}
\end{equation}
Indeed
\begin{align*}
&\textstyle{\half}\textstyle{(\ga_e-1)\one+\half a+\half e}\\
=&\textstyle{\half(\ga_e-1)\one+\half a+\half\big((2\ga_e-1)a+(1-\ga_e)\one+\gr_e b_{\gl}\big)}\\
=&\textstyle{\ga_e a+\half\gr_e b_{\gl}=ae.}
\end{align*}
Let
\[
x:= a+(1-2\ga_e)e+(\ga_e-1)\one,
\]
Note that $x=0$ precisely when $e\in\{a,\one-a\}.$
Using equation \eqref{ae} we have
\begin{align*}
ex &\textstyle{=ae+(1-2\ga_e)e+(\ga_e-1)e=\half(\ga_e-1)\one+\half a+\half e -\ga_e e}\\
& =\textstyle{\half\big((1-2\ga_e)e+a+(\ga_e-1)\one\big) =\half x.}
\end{align*}
Thus $A_1(e)=\ff e,$ $A_0(e)=\ff (\one-e),$ and $A_{\half}(e)=\ff x.$

We check the fusion rules.
\begin{gather*}
x^2=\big(a+(1-2\ga_e)e+(\ga_e-1)\one\big)^2\\
=a+(1-2\ga_e)^2e+(\ga_e-1)^2\one+2(1-\ga_e)ae+2(\ga_e-1)a+2(1-2\ga_e)(\ga_e-1)e\\
=(2\ga_e-1) a-(1-\ga_e)e+(\ga_e-1)^2\one+2(1-\ga_e)ae\\
=\textstyle{(2\ga_e-1) a+(2\ga_e-1)e+(\ga_e-1)^2\one+2(1-2\ga_e)\big(\half(\ga_e-1)\one+\half a+\half e)\big)}\\
=\ga_e(1-\ga_e)\one.
\end{gather*}
Clearly $(\one-e)x=\half x,$ so the fusion rules are satisfied and $e$ is a primitive axis in $A.$

Finally we compute that
\begin{align*}
x&=a+(1-2\ga_e)\Big((2\ga_e-1)a+(1-\ga_e)\one+\gr_eb_{\gl}\Big)+(\ga_e-1)\one\\
&=\big(1-(2\ga_e-1)^2\big)a-2\ga_e(1-\ga_e)\one+(1-2\ga_e)\gr_eb_{\gl}\\
&=4\ga_e(1-\ga_e)a-2\ga_e(1-\ga_e)\one+(1-2\ga_e)\gr_eb_{\gl}.\qedhere
\end{align*}
\end{proof}

\begin{prop}\label{phalf}$ $
\begin{enumerate}
\item
Let $e:=\ga_e a+(1-\ga_e)(\one-a)+\gr_e b_{\gl}$ be a primitive axis
in $A,$ as in~Theorem \ref{half}(i), and let $c:=a^{\gt_e},$ and set
$\mu:=2\ga_e-1.$ Then
\[
a^{\gt_e}=c=\ga_c a+(1-\ga_c)(\one-a)+2\mu\gr _e b_\gl,\quad \ga_c=\mu^2.
\]
In particular, $c_\gl = 2(2\ga_e-1)\gr _e b_\gl.$

\item
If $\ga_b\ne 0,$ then there exists a primitive axis $e\in A,$ with $a^{\gt_e}=b,$
iff $\ff$ contains a square root of $\ga_b.$

\item
If $\ga_b=0,$ let $E_d=\{a+\gr e\mid \gr\in\ff\},$ for $d\in\{a,\one-a\}.$
Then
\[
E_d=\{e_1^{\gt_e}\mid \text{$e$ a primitive idempotent in A}\},
\]
for any $e_1\in E_d.$
\end{enumerate}
\end{prop}
\begin{proof}
(1)
Using Theorem \ref{half}(i) with $e$ in place of $a$
and $a$ in place of $e,$ decompose $a$ with resprect to $e$ as
\[
a=\gvp_e(a) e+(1-\gvp_e(a))(\one-e)+a_{\gl}.
\]
Noticing that $\gvp_e(a)=\gvp_a(e)$ we have
\begin{align*}
a^{\gt_e}&=\ga_e e+(1-\ga_e)(\one-e)-\big(a-\ga_e e-(1-\ga_e)(\one-e)\big)\\
&=2(2\ga_e-1) e + 2(1-\ga_e)\one -a.
\end{align*}
Write $e=w_e+e_{\gl},$ with $w_e=\ga_ea+(1-\ga_e)(\one-a)=2(\ga_e-1)a+(1-\ga_e)\one,$ and $e_{\gl}=\gr_eb_{\gl}.$ Then
\begin{equation*}
\begin{aligned}
a^{\gt_e} & = 2(2\ga_e-1)( w_e  + e_\gl) + 2(1-\ga_e)\one -a\\
&= 2(2\ga_e-1) w_e   + 2(1-\ga_e)\one -a + 2(2\ga_e-1)e_\gl\\
&=2(2\ga_e-1)\big(2\ga_e-1)a+(1-\ga_e)\one\big) + 2(1-\ga_e)\one -a + 2(2\ga_e-1)e_\gl\\
&=(2(2\ga_e-1)^2-1)a+4(1-\ga_e)\ga_e\one+2(2\ga_e-1)e_\gl\\
&=(2\ga_e-1)^2a+(1-(2\ga_e-1)^2)(\one-a)+2(2\ga_e-1)e_\gl.
\end{aligned}
\end{equation*}
Since $e_\gl = \gr_e b_\gl,$ (1) holds.
\medskip

(2)
Suppose that $\ga_b\ne 0$ and that $\ff$ contains a square root
of $\ga_b,$ and let $\ga_e\in\ff$ with $(2\ga_e-1)^2=\ga_b.$  Then
\[
4\ga_e^2-4\ga_e+1=\ga_b,
\]
so $\ga_e-\ga_e^2=\frac{1-\ga_b}{4}.$  But then the equation
$\gr_e^2\nu_b=\ga_e(1-\ga_e)$ has a solution in $\ff,$ that is
$\gr_e=\frac{1}{2(2\ga_e-1)}.$  Hence by Theorem \ref{half},
$$e:=\ga_e a+(1-\ga_e)(\one-a)+\gr_e b_{\gl}$$ is a primitive axis in
$A,$ and $a^{\gt_e}=b,$ by (1).

The converse is easy to check using (1).
\medskip

(3)
Reversing the role of $a$ and $\one-a,$ we may assume that $d=a.$
Since $e_2=e_1+\gr b_{\gl},\ \gr\in\ff,$ for any $e_1, e_2\in E_a,$
we may assume that $e_1=a.$
By Theorem \ref{half}(i1), $\ga_e\in\{0,1\},$ for all primitive axes  $e\in A.$
Hence, in (1), $\mu^2=1,$ so  $a^{\gt_e}=a+2\gr_eb_{\gl}.$
\end{proof}

This concludes subsection \ref{sub 2-gen}.  We conclude \S 5
with the following theorem.

\begin{thm}\label{axes}
Let $A=\lan\lan a,b\ran\ran$ be a $2$-generated PAJ.  Then any
idempotent in $A$ distinct from $0$ and $\one$ (if $\one$ exists) is
a primitive axis in $A.$
\end{thm}
\begin{proof}
We take $A$ case by case.
\medskip

(1) $A=U_{\{a,b\}}(\gl).$  Then $ab=\gd a+\gl b,$ and ${ba=\gd b+\gl a.}$
By \cite[Example 2.6(1)]{RoSe1}, the idempotents in $A$ are $e_{\ga}:=\ga a+(1-\ga) b,$ $\ga\in\ff.$
Then $A_{\gl,\gd}(e_{\ga})=a-b,$ and $(a-b)^2=0,$ so $e_{\ga}$ is a primitive axis.
\medskip

(2) $A=U'_{\{a,b\}}(-1),$ with $-1\ne\half,$ as in Example \ref{Uprime}, that is $ab=-a-b.$
By \cite[Lemma 3.1.8(5)]{HSS0} the idempotents in $A$ are $a, b, -a-b$ and they are all primitive axes.
\medskip

(3) $A=A_{\operatorname{exc},3}(\{a,b\},\gl)$ as in Example
\ref{try1}. Then, by \cite[Example 2.6(2)]{RoSe1}, the non-trivial idempotents
in $A$ are $e_{\ga}=a+\ga y$ and $b+\ga y,$ $\ga\in\ff.$ Note that
$A_{\gl,\gd}(e_{\ga})=\ff y,$ for all $\ga.$  If $\ga =-1,$ then
$A_0(e_{-1})=b,$ while if $\ga\ne -1,$ then $A_0(e_{\ga})=b-z,$ with
$z=(1+\ga y).$ Clearly the fusion rules are satisfied, so $e_{\ga}$
is a primitive axis.  Similarly, $b+\ga y$ is a primitive axis in
$A.$

(4) In the notation of \cite{HRS},
$A=B(\half,1),$ or $A=B(-1,-\half),$ so $\dim(A)=3,$ and $\gc=0.$
By Lemma \ref{gc=0}, the theorem holds.
\medskip

(5)
$A=B(\gl,\gvp),$ with $\dim(A)=3,$ and $\gc\ne 0.$
By Theorems \ref{nothalf} and \ref{half}, the theorem holds.
\end{proof}

\section{The existence of a Frobenius form}\label{Frob}

In this section  $A$ is a  PAJ  generated by a set~$X$   of
primitive axes. We refer the reader to Definition \ref{hom} for the
terminology below. Let $Y$ be a connected component of the axial
graph on $\chi(A),$  then one of the following holds: \medskip

 (Case I)\phantom{II}\  $Y$ is uniform of type $\{\gl, \gd\}$ with $\gl\ne\half,$
and $\gl+\gd=1.$ This follows from Remark~\ref{ugen}.

(Case II)\phantom{I}\  $Y$ is uniform of type $\gl.$  This is always
true if there exists an axis in~$Y$ of type $\half.$ See Lemma
\ref{gr11}(iii).

(Case III)\  There exists $\half\ne\gl\in\ff$ such that
$Y=Y_{\gl}\cup Y_{1-\gl},$ and $Y_{\nu}\ne\emptyset,$ for
$\nu\in\{\gl,1-\gl\}.$  See Lemma \ref{gr11}(iii).

\begin{remark}\label{gen}
The proof of \cite[Proposition 2.12(vi)]{RoSe1} shows more than its
statement.  Namely, it shows that if $B:=\lan\lan a,b\ran\ran$ is a
CPAJ generated by two primitive axes $a$ of type $\gl$ and  $b$ of
type $\gl',$ and $\gc\ne 0$ (where $\gc$ is as in~\S 5), so
$\dim(B)=3,$ then either $\gl=\gl',$ or $B=\lan\lan a,
a^{\gt_b}\ran\ran=\lan\lan b,b^{\gt_a}\ran\ran,$ unless $\gl'=-1,
\gl=2,$ and then  $B=\lan\lan a, a^{\gt_b}\ran\ran.$ (Note that in
the proof of part (vi) of \cite[Proposition 2.12]{RoSe1}, the
displayed equalities in the proof of that part should start with
$\gl'a_{\gl'}$ and not with $\half a_{\gl'}.$)
\end{remark}

\begin{prop}\label{miya27}
Let $Y$ be a connected component of $\chi(A)$ as in Case III.
If $\gl\notin \{-1,2\},$ then $\lan\lan Y\ran\ran$
is spanned by $Y_{\nu},$
for either $\nu\in\{\gl, 1-\gl\},$ and $Y_{\nu}$ is connected and closed. If $\gl\in \{-1,2\},$ then $\lan\lan Y\ran\ran$
is spanned by $Y_2,$ and $Y_2$ is connected and closed.
\end{prop}
\begin{proof}
Clearly $Y_{\nu}$ is closed. Let $a\in Y$ be of type $\nu$ if
$\gl\notin\{-1,2\}$ and of type~$\nu=2$ otherwise. Let $y\in Y,$ and
let $a=y_1, \dots, y_m = y$ be a path connecting $a$ to $y$ in  $Y.$
We claim that there is a path $a=z_1,\dots z_{m-1}, y_m=y$ in $Y,$
such that $z_i$ is of type $\nu,$ for $i\in\{1,\dots,m-1\}.$

If $y_2$ is not of type $\nu,$ then, by Lemma \ref{dim2},
$\dim \lan\lan y_1,y_2\ran\ran =3,$ and by Remark \ref{gen},
$\lan\lan y_1,y_2\ran\ran=\lan\lan y_1, y_1^{\gt_{y_2}}\ran\ran.$
But now, by Theorem \ref{KMt}(i), $\lan\lan y_1, y_2\ran\ran,$
is spanned by axes of type $\nu,$  namely the axes $\Cl(y_1,y_1^{\gt_{y_2}})\subseteq Y.$ Note
that $y_3$ is connected to some axis
say $z_2\in \Cl(y_1,y_1^{\gt_{y_2}}),$ otherwise $y_2y_3=0.$  By induction on $m,$
the claim holds.

If $y$ is of type $\nu,$ then we found a path from $a$ to $y$ in $Y_{\nu}.$  Thus $Y_{\nu}$
is connected.  If $y$ is of type $1-\nu,$  then $\lan\lan z_{m-1}, y\ran\ran$ is spanned
by $\Cl\{z_{m-1}, z_{m-1}^{\gt_y}\},$ and we see that $y$ is in the span of $Y_{\nu}.$
\end{proof}

Let $\{X_i\mid i\in I\}$ be the set of connected components of $\chi(A).$
If $X_i$ is as in Case I or Case II, let $X'_i=X_i.$
If $X_i$ is as in case III, and $\gl\in\{-1,2\}$ let $X'_i=(X_i)_2.$
Finally if  $X_i$ is as in case III, and $\gl\notin\{-1,2\},$
choose $\gl_i\in \{\gl,1-\gl\}$ and let $X'_i=(X_i)_{\gl_i}.$
Let $X': = \bigcup X_i'\subseteq X$.

Of course, using Proposition \ref{miya27}, we aee that $X'$ spans $A,$ and $X'=\Cl(X').$
Let $\calb = \cup \calb_i$ be a basis of $A$ with
 $\calb_i \subseteq X_i'.$

 Recall that a non-zero bilinear form
$(\cdot\, ,\, \cdot)$ on an algebra $A$ is called {\bf Frobenius} if
the form associates with the algebra product, that is,
\[(x,yz)=(xy,z)\]
for all $x,y,z\in A$.

For PAJs we specialize the concept of Frobenius form further by
asking that the ``normalized'' condition $(a,a)=1$ be satisfied for
each axis $a\in X'.$

 The purpose of this section is to prove
the following theorem (extending \cite[Theorem 4.1, p.~407]{HSS}):

\begin{thm}\label{thm frob}
The algebra $A$ admits a normalized Frobenius form, which is unique
up to choice of $\gl_i.$
The form $(\ ,\ )$ satisfies  $( a,u) = \gvp_a(u)$ for all
$a\in X'$ and $u\in A.$
\end{thm}

\begin{proof}[Proof of Theorem \ref{thm frob}]\hfill
\medskip

We start by defining the bilinear form $(\cdot\, ,\, \cdot)$ on the
 base $\calb\subseteq X'.$ We let $(a,b)=\varphi_a(b),$  for all $a,b\in\mathcal{B}.$
Extending by linearity we get the bilinear form $(\cdot\, ,\,
\cdot)$. Note that by our choice of $\mathcal{B}$ and by
Proposition~\ref{sym}, $(\cdot\, ,\, \cdot)$ is symmetric.

\begin{lemma}
\begin{enumerate}
\item
$(a\, ,\, u)=\gvp_a(u),$ for all axes $a\in X'$ and all $u\in A;$

\item
$(a\, ,\, a)=1,$ for all axes $a\in X';$

\item
$(\cdot\, ,\, \cdot)$ is invariant under any automorphism of $A$
that preserves $X'.$ In particular
$(\cdot\, ,\, \cdot)$ is invariant under any automorphism in $\calg(X').$
\end{enumerate}
\end{lemma}
\begin{proof}
(1\&2): First suppose that $a\in\calb$. We know that (1) holds for
any $u$ in $\calb,$ by definition.  For $u =  \sum_{b\in\calb}\ga_b
b ,$ since $\gvp_a$ is linear,
\begin{gather*}
\gvp_a(u)=
\gvp_a\left(\sum_{b\in\calb}\ga_b b\right)=\sum_{b\in\calb}\ga_b\gvp_a(b)\\
=\sum_{b\in\calb}\ga_b(a,b) =(a,\sum_{b\in\calb}\ga_b b)=(a,u),
\end{gather*}
and (1) holds for $a\in \calb$, $u\in A.$

Now  suppose that $a\in X' \setminus \calb$.  Let $b \in \calb$.
Then $\gvp_a(b)=\gvp_b(a),$ by the choice of $X'.$ Further
by~Proposition~\ref{sym}, and $\gvp_b(a)=(b,a),$ as (1) holds
for~$b$.  Finally, since $(\cdot\, ,\, \cdot)$ is symmetric
$(b,a)=(a,b),$ so $\gvp_a(b)=(a,b),$ so, as above,
$(a,u)=\gvp_a(u),$ for all $u\in A,$ and (1) holds for any axis
$a\in X'$.

In particular, $(a,a)=1$ for every  axis $a\in X'$,  since
$\gvp_a(a)=1$. Thus (2) holds.
\medskip

\noindent (3):\quad Let $\psi\in\Aut(A)$ preserving $X',$ and $a\in
X'$ be a primitive axis of type $(\gl,\gd)$  or $\gl.$ If
\[
u=\gvp_a(u)a+u_0+u_{\gl,\gd}
\]
is the decomposition of $u\in A$
with respect to  $a,$ then  the decomposition of $u^\psi$ with
respect to the primitive axis $a^\psi\in X'$, also of type $(\gl,\gd)$,
is $u^\psi=\gvp_a(u)a^\psi+u_0^\psi+u_{\gl,\gd}^\psi$. Hence
$\gvp_{a^\psi}(u^\psi)=\gvp_a(u)$, and so $(a^\psi,u^\psi)=(a,u)$.
Finally, taking an arbitrary $v\in A$ and decomposing it with
respect to the base $\calb$ as $v=\sum_{b\in\calb}\ga_b b$, we get
that
\begin{equation}
\begin{aligned}
(v^\psi,u^\psi) & =\big(\sum_{b\in\calb}\ga_b b^\psi,u^\psi\big)=
\sum_{b\in\calb}\ga_b(b^\psi,u^\psi)\\ &=\sum_{b\in\calb}\ga_b(b,u)=
\big(\sum_{b\in\calb}\ga_b b,u\big)=(v,u).\qedhere
\end{aligned}
\end{equation}
\end{proof}

\begin{lemma}\label{O}
For every  axis $a\in X'$, the different eigenspaces of $a$ are
orthogonal with respect to $(\cdot\, ,\, \cdot)$.
\end{lemma}
\begin{proof}
We assume that $a$ is of type $(\gl,\gd).$  The
argument for $a$ of type $\gl$ is exactly the same.

Clearly, if $u\in A_0(a)+A_{\gl,\gd}(a),$ then $(a,u)=\gvp_a(u)=0$.
Hence $A_1(a)=\ff a$ is orthogonal to both $A_0(a)$ and
$A_{\gl,\gd}(a)$. It remains to show that $A_0(a)$ orthogonal to
$A_{\gl,\gd}(a)$. For $u\in A_0(a)$ and $v\in A_{\gl,\gd}(a)$, the
fact that $(\cdot\, ,\, \cdot)$ is invariant under $\tau_a$ gives us
$(u,v)=(u^{\tau_a},v^{\tau_a})=(u,-v)=-(u,v)$. Clearly, this means
that $(u,v)=0$.
\end{proof}

\begin{lemma}
Let $a\in X'$ be a primitive axis of type $(\gl,\gd),$ with $\gl\ne\gd,$
or of type $\gl.$ Let  $v\in A$
be a $\mu$-eigenvector of $a$ and $w$ a $\nu$-eigenvector of $a,$
with $\mu,\nu\in\{1,0,(\gl,\gd)\},$ (resp.~$\mu,\nu\in\{0,1,\gl\}$). Then
\begin{enumerate}
\item
If $\mu\ne\nu,$ then $0=(va,w)=(av,w)=(v,aw)=(v,wa)$.

\item
If $\mu=\nu\in\{0,1\},$ then $(va,w)=(av,w)=(v,aw)=(v,wa).$

\item
$(x,yz)=(yx,z),$ and $(xy,z)=(x,zy),$ for all $x,y,z\in A.$

\item
If $\gl\ne\gd,$ and $\mu=\nu=(\gl,\gd),$ then $(v,w)=0.$

\item
$(x,yz)=(xy,z),$ for all $x,y,z\in A.$
\end{enumerate}
\end{lemma}
\begin{proof}
(1)  $(va,w)=\gr(v,w),$ for some  $\gr\in\ff,$ so by Lemma \ref{O},
$(va,w)=0.$  Similarly $(av,w)=(v,aw)=(v,wa)=0.$
\medskip

\noindent (2)  This is clear.
\medskip

\noindent (3)
Consider the equality $(x,yz)=(yx,z).$ This is linear
in $x$, $y$, and $z$. In particular, since $A$ is spanned by $X',$
we may assume that $y:=a\in X'.$
Furthermore, since $A$ decomposes as the sum of the eigenspaces of
$a$, we may assume that $x:=v$ and $y:=w$ are eigenvectors of~$a.$
By (1) and (2) we  only consider the case where $v$ and $w$ are
$(\gl,\gd)$-eigenvectors (resp.~$\gl$-eigenvectors), but clearly the equality holds here. The
proof of the second equality in (3) is similar.
\medskip

\noindent (4)  By (3) we have
\[
\gl(v,w)=(v,aw)=(av,w)=(a,wv)=(wa,v)=(v,wa)=\gd(v,w).
\]
Since $\gl\ne\gd,$ $(v,w)=0.$
\medskip

\noindent (5)
Again we may assume that $y\in X'$ is an
axis and that $x,
z$ are eigenvectors of~$y.$ Then (5) follows from (1), (2) and (4).
\end{proof}
Part (5) concludes the proof of Theorem \ref{thm frob}.
\end{proof}

\begin{remark}\label{-1}
Notice that by our choice of $X'$ in Theorem \ref{thm frob},
for any primitive axis $a\in A,$ we can choose $X'$ so that
$a\in X',$ unless perhaps $a$ is of type $-1\ne\half$ and there
exist a primitive axis $b\in A$ of type $2,$ with $ab\ne 0.$
\end{remark}

The normalized Frobenius form of a noncommutative uniformly
generated PAJ is easy to compute.

\begin{examples}\label{Frobnoncom}
Suppose that $A$ is as in Theorem~\ref{almostcomm}, a uniformly generated PAJ of type $\{\gl,\gd\}$, where $\gl
+ \gd = 1$ and $\gl \ne \gd.$  Then for $a,b\in\chi(A),$
\[
(a,b) = \begin{cases}
1 &\text{if $a,b$ are axes of the same type in the same component};\\
0 &\text{otherwise.}
\end {cases}
\]
Indeed, note first that in this case $X'$ of Theorem \ref{thm frob} equals $\chi(A).$
If $a,b$ are of the same type, and in the same component,
then, by~Theorem \ref{U1}(1), $\lan\lan a,b\ran\ran\cong U_2(\gl),$
so $b=a+b-a,$ with $b-a \in A_0(a).$ Thus $(a,b)=\gvp_a(b)=1.$
Otherwise $ab=0.$  If $a$ has type $(\gl,\gd)$ and $b$ has type
$(\gd,\gl)$ then either $ab = 0$ so $(a,b) = 0,$ or $b = (b-y)+y$ in
$A_{\operatorname{exc},3}(\{a,b\},\gl)$, so $(a,b) =\gvp_a(b) = 0.$
\end{examples}

Recall that the \textbf{radical} of an associative bilinear form on
an algebra $A$ is $\{ u\in A: (a,A) = 0\}$, and is an ideal of $A.$

\begin{remark}\label{rad}
Let $A$ be a PAJ and let $(\ ,\ )$ be a Frobenius
form on $A.$  Let $b\in A$ be a primitive axis of type $(\gl,\gd)$
(we allow here $\gl=\gd$).
If $(b,b)=0,$ then $b$ is in the radical $R$ of $(\ ,\ ).$
Indeed, let $y\in A_0(b).$  Then $(b,y)=(b^2,y)=(b,by)=(b,0)=0.$
Let $z\in A_{\gl,\gd}(b).$  Then $(b,z)=(b^2,z)=(b,bz)=\gl (b,z),$
so $(b,z)=0.$  Since $A=\ff b+A_0(b)+A_{\gl,\gd}(b),$
if $(b,b)=0,$ then $(b,A)=0,$ so $b\in R.$
\end{remark}

\begin{lemma}
The radical  of the Frobenius form on a uniformly generated PAJ $A$ of type $\{ \gl, \gd\},$
with $\gl+\gd=1,$ and $\gl\ne\gd,$
equals
$R =\sum \ff(a -b),$
where the sum is taken over all pairs $a,b$ of primitive axes  of the same type,
and in the same connected component.
$A/R$ is a direct product of copies of $\ff.$
\end{lemma}
\begin{proof} Clearly $R$ is contained in the radical, by
Examples~\ref{Frobnoncom}.  By Theorem \ref{almostcomm}(ix), $A/R$ is a direct product of copies
of $\ff,$ so $R$ is the radical.
\end{proof}

\begin{thm} \label{A02}
Let $a\in A$ be a primitive axis. Then either
\begin{itemize}
\item[(i)]
$A_0(a)^2\subseteq A_0(a);$ or

\item[(ii)]
$a$ is of type $-1\ne\half,$  there exists a primitive axis
$b\in A$ of type $2$ with $ab\ne 0,$
and $a$ is in the radical of any Frobenius form on $A.$
\end{itemize}
\end{thm}
\begin{proof}
If $a$ is not as in (ii),
then using Remark \ref{-1},
we  choose a Frobenius form such that $a\in X'.$ Let $x,y\in A_0(a).$  Then
\[
(a,xy)=(ax,y)=0.
\]
Since $xy\in \ff a+A_0(a),$ and $(a, A_0(a))=0,$ while $(a,a)=1,$
we must have $xy\in A_0(a).$

Now suppose that $a$ is as in (ii), and let $(\ ,\ )$ be
any Frobenius form on $A.$
   If $(a,a)\ne 0,$ then the same argument as
in the previous paragraph of the proof shows that $A_0(a)^2\subseteq A_0(a).$
Hence  $(a,a)=0,$ so by Remark \ref{rad}, $a$ is in the radical of $(\ ,\ ).$
\end{proof}

\begin{remark}
Let $A=B(2,\frac{3}{2}),$ as in Remark \ref{B}, with $2\ne -1.$
So $A$ is generated by primitive axes $a,$ of type $2$ and $b$ of type
$-1.$  One can check that for the Frobenius form $(\ ,\ )$
that we constructed in Theorem \ref{thm frob}, $(b,b)=0,$
so by Remark \ref{rad}, $b$ is in the Radical of that form.
However the following question remains open:
\medskip

\noindent
{\bf Question.}
\smallskip

\noindent
Let $A$ be a CPAJ and let $a\in A$
be a primitive axis as in Theorem \ref{A02}(ii).
Is it true that $A_0(a)^2\subseteq A_0(a)?$
\end{remark}


\end{document}